\documentclass[10pt,authoryear]{article}%
\usepackage{amsmath}
\usepackage{amsfonts}
\usepackage{mathrsfs}
\usepackage{amssymb,bbm, color}
\usepackage[hidelinks,colorlinks=true,linkcolor=blue,citecolor=blue]{hyperref}
\usepackage{graphicx}
\numberwithin{equation}{section}
\usepackage[body={15.5cm,21cm}, top=3cm]{geometry}%
\providecommand{\U}[1]{\protect\rule{.1in}{.1in}}
\providecommand{\U}[1]{\protect \rule{.1in}{.1in}}
\newtheorem{theorem}{Theorem}[section]

\newtheorem{assumption}[theorem]{Assumption}

\newtheorem{definition}[theorem]{Definition}

\newtheorem{lemma}[theorem]{Lemma}

\newtheorem{proposition}[theorem]{Proposition}
\newtheorem{remark}[theorem]{Remark}

\newenvironment{proof}[1][Proof]{\noindent \textbf{#1.} }{\  \rule{0.5em}{0.5em}}

\def \P{\mathbb{P}}
\def \E{\mathbb{E}}
\def \hE{\hat{\mathbb{E}}}

\begin{document}
	\title{Doubly Reflected Backward SDEs Driven by $G$-Brownian Motion with Quadratic Generator}
	\author{ Hanwu Li\thanks{Research Center for Mathematics and Interdisciplinary Sciences, Shandong University, Qingdao 266237, Shandong, China. lihanwu@sdu.edu.cn.}
	\thanks{Frontiers Science Center for Nonlinear Expectations (Ministry of Education), Shandong University, Qingdao 266237, Shandong, China.}
    \thanks{Shandong Province Key Laboratory of Financial Risk, Shandong University, Qingdao 266237, Shandong, China.}
    , Peng Luo\thanks{School of Mathematical Sciences, Shanghai Jiao Tong University, 200240, Shanghai, China. peng.luo@sjtu.edu.cn}
    \hspace{0.01em} and Mengbo Zhu\thanks{Zhongtai Securities Institute for Financial Studies, Shandong University, Jinan 250100, Shandong, China. zhumengbo@mail.sdu.edu.cn} } 
	\date{}
	\maketitle
	\begin{abstract}
	In this paper, we study the doubly reflected  backward stochastic differential equations driven by $G$-Brownian motion ($G$-BSDEs for short) when the generator has quadratic growth in the $z$-component. Based on the theory of $G$-BMO martingale and $G$-Girsanov theorem, we establish the existence and uniqueness result when the upper obstacle is almost a generalized $G$-It\^{o}'s process. Moreover, the solution can be approximated monotonically by the solutions to a family of penalized reflected $G$-BSDEs with a lower obstacle, which plays an important role to establish the relation between doubly reflected $G$-BSDEs and fully nonlinear partial differential equations with double obstacles.
	\end{abstract}
	
	
	\textbf{MSC-classification}: 60H10, 60H30
	
	\section{Introduction}

The theory of nonlinear backward stochastic differential equations (BSDEs for short) was introduced by Pardoux and Peng \cite{PP}. Later, El Karoui et al. \cite{EPQ} extended their results to reflected BSDEs with a lower obstacle of the following form
	\begin{equation}
		\begin{cases}			
			Y_t=\xi+\int_t^T f(s,Y_s,Z_s)ds-\int_t^T Z_sdB_s+(A_T-A_t), \\
			Y_t\geq L_t,\ t\in[0,T],\ \textrm{and} \ \int_0^T (Y_s-L_s)dA_s=0,
		\end{cases}
    \nonumber
	\end{equation}
where the generator $f$ is uniformly Lipschitz in $(y,z)$. The non-decreasing process $A$ is added to the original BSDEs to keep the solution $Y$ above the prescribed obstacle process $L$ in a manner that satisfies the Skorohod condition $\int_0^T (Y_s-L_s)dA_s=0$. 
Subsequently, Cvitani\'{c} and Karatzas \cite{CK} proposed BSDEs with two obstacles, where two non-decreasing processes $A^+,\,A^-$  are introduced to ensure that the value process $Y$ is restricted between the lower obstacle $L$ and the upper obstacle $U$ in a minimal way (i.e. $\int_0^T (Y_s-L_s)dA_s^+=0,\ \int_0^T (U_s-Y_s)dA_s^-=0$). The reflected BSDE is a powerful tool to study several problems related to mathematical ﬁnance, stochastic control and PDEs. We refer the readers to \cite{BCFE,CR,EPQ} for more details.

To deal with problems with model uncertainty, Peng \cite{P07a,P08a} systemically established the $G$-expectation theory. The theory defines a new type of Brownian motion termed $G$-Brownian motion under nonlinear expectation. Compared with the classical case, its increments obey the $G$-normal distribution and {its quadratic process} is no longer deterministic. The $G$-It\^{o}'s formula is also proposed for stochastic calculus under $G$-expectation framework. 

More recently, Hu et al. \cite{HJPS1} introduced the BSDEs driven
by $G$-Brownian motion ($G$-BSDEs for short) which take the following form
	\begin{equation}			
			Y_t=\xi+\int_t^T f(s,Y_s,Z_s)ds+\int_t^T g(s,Y_s,Z_s)d\langle B \rangle_s-\int_t^T Z_sdB_s-(K_T-K_t),
    \nonumber
	\end{equation}
where the generators $f,g$ are uniformly Lipschitz in $(y,z)$ {and $K$ is a non-increasing $G$-martingale}. {This new feature causes }difficulties in constructing contraction mapping when establishing the existence and uniqueness result. A new method, called Galerkin approximation, was introduced for the issue. Their subsequent work on the comparison theorem, the Girsanov transformation and
the nonlinear Feynman–Kac formula  have been done in \cite{HJPS2}. The reflected $G$-BSDEs with a lower obstacle and with an upper obstacle were studied by Li et al. \cite{LPSH} and Li and Peng \cite{LP}, respectively. It is worth pointing out that, different from the classical framework, the case of lower obstacles and of upper obstacles are significantly different under $G$-expectation due to the appearance of the non-increasing $G$-martingale. Moreover, the minimal fashion of the compensating term is characterized by the so called martingale condition instead of the Skorohod condition.  Subsequently, Li and Song \cite{LS} proposed $G$-BSDEs with double reflections, where the description of the solutions is motivated by the penalization method. Recently, Li and Ning \cite{LN} relaxed the conditions of obstacles in \cite{LS}. The $G$-BSDEs and reflected $G$-BSDEs have wide applications including utility theory (\cite{EJ2,LTT}), optimal stopping under nonlinear expectations (\cite{Li25}), pricing contingent claims when the financial market has volatility uncertainty (\cite{EJ1,LPSH,Vorbrink}) and probabilistic interpretation for solutions to fully nolinear PDEs (\cite{HJPS2,LN,LPSH}). 

Due to the importance in theoretical analysis and practical applications, the BSDEs and reflected BSDEs under both the classical and the $G$-expectation framework have attracted considerable attention. One branch is to relax the Lipschitz condition on generators, for instance, the case where the generators have quadratic growth in the $z$-component. Quadratic BSDEs and quadratic reflected BSDEs have been extensively and intensively studied since the work of Kobylanski \cite{K}. We refer the reader to \cite{BH1,BH2} for subsequent works of \cite{K} and \cite{KLQT,LX} for quadratic reflected BSDEs. In addition to their fruitful applications in exponential utility theory, including exponential utility maximization (\cite{HIM,M}) and stochastic equilibrium  (\cite{XZ}), they are also closely related to pricing contingent claims in incomplete markets (\cite{RE}), risk-sensitive control (\cite{HTdiagbsde,LZ,LZ24}), stochastic representations for PDEs (\cite{BY,BC}). 
 In the $G$-expectation framework, Hu et al. \cite{HLSH} proved existence and uniqueness of quadratic $G$-BSDEs with bounded terminal values. Uniqueness is obtained by techniques of the $G$-BMO martingale and $G$-Girsanov theory. For the existence, first, the solutions to discrete $G$-BSDEs are constructed with the help of fully nonlinear PDEs, and then the solutions to the general case {are obtained by successive approximation}. The method is valid when the generators and terminal values are sufficiently regular. Then, Hu et al. \cite{HTW} extended the quadratic $G$-BSDEs both to the case of convex generators and unbounded terminal values and to the multi-dimensional case.  Cao and Tang \cite{CT} firstly tackled the quadratic reflected $G$-BSDEs. They also established the comparison theorem as well as the nonlinear Feynman–Kac formula.

In this paper, we study the doubly reﬂected $G$-BSDEs of the following form:
	\begin{equation}
		\begin{cases}			
			Y_t=\xi+\int_t^T f(s,Y_s,Z_s)ds+\int_t^T g(s,Y_s,Z_s)d\langle B \rangle_s-\int_t^T Z_sdB_s 
             +(A_T-A_t),   \\     L_t\leq Y_t \leq U_t,\ t\in[0,T],
		\end{cases}
    \nonumber
	\end{equation}
where the generators $f,g$ have quadratic growth in the $z$-component. Our contributions are mainly twofold. First, we attain the well-posedness of doubly reflected $G$-BSDEs with quadratic growth and second, we establish the relation between quadratic doubly reflected $G$-BSDEs and fully nonlinear partial differential equations with two obstacles. 
Compared with the approximate Skorohod condition with a fixed order $\alpha$ proposed in \cite{LN,LS}, for the quadratic case, similar condition must hold for any order $\alpha\geq 2$.  By employing the properties of $G$-BMO martingales and $G$-Girsanov theory, we establish several a priori estimates that are crucial to the uniqueness of the solution. The existence is attained by the approximation via penalization. In order to derive the connection between the solution to the doubly reflected $G$-BSDE and the PDE with two obstacles, it is more suitable to consider the penalized reflected $G$-BSDEs with a lower obstacle (see Eq. \eqref{barY^n}) as the associated sequence of solutions is monotone. However, even for the Lipschitz case, this approach is also hardly feasible due to the fact that $\bar{A}^n$ in \eqref{barY^n} is no longer a $G$-martingale as claimed in \cite{LN}. Therefore, instead of the penalized reflected $G$-BSDEs, we consider the following family of $G$-BSDEs parameterized by $m,n$:
\begin{equation}
	\begin{split}
		Y^{n,m}_t=&\xi+\int_t^T f(s,Y^{n,m}_s,Z^{n,m}_s)ds+\int_t^T g(s,Y^{n,m}_s,Z^{n,m}_s)d\langle B\rangle_s-\int_t^T Z_s^{n,m}dB_s\\
        &-(K_T^{n,m}-K_t^{n,m}) +(A^{n,m,+}_T-A^{n,m,+}_t)-(A^{n,m,-}_T-A^{n,m,-}_t),
\end{split} \nonumber \end{equation}
where $A^{n,m,+}_t=\int_0^t m(Y_s^{n,m}-L_s)^-ds$ and $A^{n,m,-}_t=\int_0^t n(Y_s^{n,m}-U_s)^+ds$. We first provide the uniform estimate for the sequence $\{Y^{n,m}\}$ by applying the comparsion theorem for quadratic $G$-BSDEs. Recall that for the lower obstacle case, the uniform estimates in terms of $Z$ and $K$ can be dominated by those with respect to $Y$ (see Proposition 3.1 and Lemma 4.1 in \cite{CT}). Unfortunately, since the present penalized $G$-BSDEs involve two competing forces, for the doubly reflected case, uniform estimates for $\{Z^{n,m}\}$, $\{K^{n,m}\}$, $\{A^{n,m,\pm}\}$ cannot be derived solely from uniform estimates for $\{Y^{n,m}\}$. To overcome this shortcoming, we provide the explicit convergence rate of $(Y^{n,m}-U)^{+}$, which is attained by some delicate analysis  using the comparison theorem for quadratic $G$-BSDEs, the theory of $G$-BMO maringales and $G$-Girsanov transformation. With the help of estimates mentioned above, choosing $m=n$, we show that the solutions of penalized $G$-BSDEs converge when $n$ approaches infinity in some suitable spaces, and the triple of limiting processes is indeed  the solution to the doubly reflected $G$-BSDE. 

Although we have already obtained the existence of the solution to doubly reflected $G$-BSDE, the sequence $\{Y^{n,n}\}$ does not converge to $Y$ monotonically. Recall that the Dini theorem plays an important role in uncovering the relation between  reflected $G$-BSDEs (resp., reflected BSDEs) and obstacle problems for fully nonlinear (resp., semilinear) PDEs (see \cite{KKPPQ,HH,LN,LPSH}). Thus, it is crucial for us to find a monotone convergence sequence. Fortunately, sending $m$ to infinity in the above penalized $G$-BSDEs, we obtain a family of penalized reflected $G$-BSDEs with a lower obstacle parameterized by $n$, whose solutions are denoted by $\{(\bar{Y}^n,\bar{Z}^n,\bar{A}^n)\}$.  Since $\{\bar{Y}^n\}$ is non-decreasing with respect to $n$, it is exactly the monotonic sequence we need providing that $|Y^{n,n}-\bar{Y}^n|$ converges to $0$. With the help of this fact and the connection between reflected $G$-BSDEs with a lower obstacle and the PDEs with a single obstacle, we finally obtain the Feynman-Kac formula for the doubly reflected $G$-BSDEs under a Markovian framework. 


The remainder of the paper is organized as follows. In Section \ref{sec:Preliminaries}, we present some preliminary results on $G$-expectation, quadratic $G$-BSDEs and quadratc reflected $G$-BSDEs with a lower obstacle. Then,  we 
introduce the quadratic doubly reflected $G$-BSDE and establish its  well-posedness in Section \ref{sec:Doubly reflected $G$-BSDE with quadratic generator}.  In Section \ref{sec:Probabilistic representation for fully nonlinear PDEs with double obstacles}, we establish the relationship between doubly reflected $G$-BSDEs with quadratic growth and fully nonlinear parabolic PDEs.

\section{Preliminaries} \label{sec:Preliminaries}

We review some basic notions and results of $G$-expectation, $G$-BSDEs and reflected $G$-BSDEs. For simplicity, we only consider the one-dimensional case. The readers may refer to  \cite{CT,HLSH,HTW,LN,P07a,P08a,P19} for more details.

\subsection{$G$-expectation and $G$-It\^{o}'s calculus}

Let $\Omega_T=C_{0}([0,T];\mathbb{R})$, the space of
real-valued continuous functions starting from the origin, i.e., $\omega_0=0$ for any $\omega\in \Omega_T$, be endowed
with the supremum norm. Let $\mathcal{B}(\Omega_T)$ be the Borel set and $B$ be the canonical process. Set
\[
L_{ip} (\Omega_T):=\{ \varphi(B_{t_{1}},...,B_{t_{n}}):  \ n\in\mathbb {N}, \ t_{1}
,\cdots, t_{n}\in\lbrack0,T], \ \varphi\in C_{b,Lip}(\mathbb{R}^{ n})\},
\]
where $C_{b,Lip}(\mathbb{R}^{ n})$ denotes the set of all bounded Lipschitz functions on $\mathbb{R}^{n}$.
We fix a sublinear and monotone function $G:\mathbb{R}\rightarrow\mathbb{R}$ defined by
\begin{align}\label{GG}
G(a):=\frac{1}{2}(\overline{\sigma}^2a^+-\underline{\sigma}^2a^-), \quad \textrm{for} \quad a \in \mathbb{R},
\end{align}
where $0< \underline{\sigma}^2<\overline{\sigma}^2$. The related $G$-expectation on $(\Omega_T, L_{ip}(\Omega_T))$ can be constructed in the following way. Assume that $\xi\in L_{ip}(\Omega_T)$ can be represented as
    \begin{displaymath}
    	\xi=\varphi(B_{{t_1}}, B_{t_2},\cdots,B_{t_n}).
\end{displaymath}
Then, for $t\in[t_{k-1},t_k)$, $k=1,\cdots,n$, set
\begin{displaymath}
	\hat{\mathbb{E}}_{t}[\varphi(B_{{t_1}}, B_{t_2},\cdots,B_{t_n})]:=u_k(t, B_t;B_{t_1},\cdots,B_{t_{k-1}}),
\end{displaymath}
where $u_k(t,x;x_1,\cdots,x_{k-1})$ is a function of $(t,x)$ parameterized by $(x_1,\cdots,x_{k-1})$ such that it solves the following fully nonlinear PDE defined on $[t_{k-1},t_k)\times\mathbb{R}$:
\begin{displaymath}
	\partial_t u_k+G(\partial_x^2 u_k)=0
\end{displaymath}
with terminal conditions
\begin{displaymath}
	u_k(t_k,x;x_1,\cdots,x_{k-1})=u_{k+1}(t_k,x;x_1,\cdots,x_{k-1},x), \quad k<n
\end{displaymath}
and $u_n(t_n,x;x_1,\cdots,x_{n-1})=\varphi(x_1,\cdots,x_{n-1},x)$. Hence, the $G$-expectation of $\xi$ is $\hat{\mathbb{E}}_0[\xi]$ and for simplicity, we always omit the subscript $0$. The triple $(\Omega_T, L_{ip}(\Omega_T),\hat{\mathbb{E}})$ is called the $G$-expectation space and the process $B$ is the $G$-Brownian motion. 

Define $\Vert\xi\Vert_{L_{G}^{p}}:=(\hat{\mathbb{E}}[|\xi|^{p}])^{1/p}$ for $\xi\in L_{ip}(\Omega_T)$, $p\geq1$ and $\|\xi\|_{L^\infty_G}:=\inf\{M\geq 0:|\xi|\leq M, \textrm{q.s.}\}$.   The completion of $L_{ip} (\Omega_T)$ under the norm $\|\cdot\|_{L_G^p}$ (resp., $\|\cdot\|_{L_G^\infty}$) is denote by $L_{G}^{p}(\Omega_T)$ (resp., $L^\infty_G(\Omega_T)$). For all $t\in[0,T]$, $\hat{\mathbb{E}}_t[\cdot]$ is a continuous mapping on $L_{ip}(\Omega_T)$ w.r.t the norm $\|\cdot\|_{L_G^1}$. Hence, the conditional $G$-expectation $\mathbb{\hat{E}}_{t}[\cdot]$ can be
extended continuously to the completion $L_{G}^{1}(\Omega_T)$. Denis, Hu and Peng \cite{DHP11} prove that the $G$-expectation has the following representation.
\begin{theorem}[\cite{DHP11}]
	\label{the1.1}  There exists a weakly compact set
	$\mathcal{P}$ of probability
	measures on $(\Omega_T,\mathcal{B}(\Omega_T))$, such that
	\[
	\hat{\mathbb{E}}[\xi]=\sup_{\P\in\mathcal{P}}\E^{\P}[\xi] \text{ for all } \xi\in  {L}_{G}^{1}{(\Omega_T)}.
	\]
	$\mathcal{P}$ is called a set that represents $\hat{\mathbb{E}}$.
\end{theorem}

Let $\mathcal{P}$ be a weakly compact set that represents $\hat{\mathbb{E}}$.
For this $\mathcal{P}$, we define the capacity%
\[
c(A):=\sup_{P\in\mathcal{P}}P(A),\ A\in\mathcal{B}(\Omega_T).
\]
A set $A\in\mathcal{B}(\Omega_T)$ is called polar if $c(A)=0$.  A
property holds $``quasi$-$surely"$ (q.s.) if it holds outside a
polar set. In the following, we do not distinguish two random variables $X$ and $Y$ if $X=Y$, q.s.

\begin{definition}
	\label{def2.6} Let $M_{G}^{0}(0,T)$ be the collection of processes in the
	following form: for a given partition $\{t_{0},\cdot\cdot\cdot,t_{N}\}$ of $[0,T]$,
	\[
	\eta_{t}(\omega)=\sum_{j=0}^{N-1}\xi_{j}(\omega)\mathbf{1}_{[t_{j},t_{j+1})}(t),
	\]
	where $\xi_{i}\in L_{ip}(\Omega_{t_{i}})$, $i=0,1,2,\cdot\cdot\cdot,N-1$. For each
	$p\geq1$ and $\eta\in M_G^0(0,T)$, let $\|\eta\|_{H_G^p}:=\{\hat{\mathbb{E}}[(\int_0^T|\eta_s|^2ds)^{p/2}]\}^{1/p}$, $\Vert\eta\Vert_{M_{G}^{p}}:=(\mathbb{\hat{E}}[\int_{0}^{T}|\eta_{s}|^{p}ds])^{1/p}$ and denote by $H_G^p(0,T)$,  $M_{G}^{p}(0,T)$ the completion
	of $M_{G}^{0}(0,T)$ under the norm $\|\cdot\|_{H_G^p}$, $\|\cdot\|_{M_G^p}$, respectively.
\end{definition}

We denote by $\langle B\rangle$ the quadratic variation process of the $G$-Brownian motion $B$. For two processes $ \xi\in M_{G}^{1}(0,T)$ and $ \eta\in M_{G}^{2}(0,T)$,
the $G$-It\^{o} integrals $(\int^{t}_0\xi_sd\langle
B\rangle_s)_{0\leq t\leq T}$ and $(\int^{t}_0\eta_sdB_s)_{0\leq t\leq T}$ are well defined, see  Li and Peng \cite{lp} and Peng \cite{P19}. The following proposition can be regarded as the Burkholder--Davis--Gundy inequality under $G$-expectation framework.
\begin{proposition}[\cite{HJPS2}]\label{BDG}
	If $\eta\in H_G^{\alpha}(0,T)$ with $\alpha\geq 1$ and $p\in(0,\alpha]$, then we have
	\begin{displaymath}
		\underline{\sigma}^p {c_{p,T}}\hat{\mathbb{E}}_t\bigg(\int_t^T |\eta_s|^2ds\bigg)^{p/2}\leq
		\hat{\mathbb{E}}_t\bigg[\sup_{u\in[t,T]}\bigg|\int_t^u\eta_s dB_s\bigg|^p\bigg]\leq
		\bar{\sigma}^p C_{p,T}\hat{\mathbb{E}}_t\bigg(\int_t^T |\eta_s|^2ds\bigg)^{p/2},
	\end{displaymath}
	where $0<c_{p,T}<C_{p,T}<\infty$ are constants depending on $p, T$.
\end{proposition}

Let $$S_G^0(0,T)=\Big\{h(t,B_{t_1\wedge t}, \ldots,B_{t_n\wedge t}):t_1,\ldots,t_n\in[0,T],\; h\in C_{b,Lip}(\mathbb{R}^{n+1})\Big\}.$$ For $p\geq 1$ and $\eta\in S_G^0(0,T)$, set $\|\eta\|_{S_G^p}=\{\hat{\mathbb{E}}[\sup_{t\in[0,T]}|\eta_t|^p]\}^{1/p}$. Denote by $S_G^p(0,T)$ (resp., $S^\infty_G(0,T)$) the completion of $S_G^0(0,T)$ under the norm $\|\cdot\|_{S_G^p}$ (resp., $\|\cdot\|_{S_G^\infty}$, where $\|\eta\|_{S^\infty_G}:=\|\sup_{t\in[0,T]}|\eta_t|\|_{L^\infty_G}$). For $\xi\in L_{ip}(\Omega_T)$, let $\mathcal{E}(\xi)=\hat{\mathbb{E}}[\sup_{t\in[0,T]}\hat{\mathbb{E}}_t[\xi]]$. For $p\geq 1$ and $\xi\in L_{ip}(\Omega_T)$, define $\|\xi\|_{p,\mathcal{E}}=[\mathcal{E}(|\xi|^p)]^{1/p}$ and denote by $L_{\mathcal{E}}^p(\Omega_T)$ the completion of $L_{ip}(\Omega_T)$ under $\|\cdot\|_{p,\mathcal{E}}$. Similar to the classical Doob's maximal inequality, the following theorem holds.
\begin{theorem}[\cite{S11}]\label{the1.2}
	For any $\alpha\geq 1$ and $\delta>0$, $L_G^{\alpha+\delta}(\Omega_T)\subset L_{\mathcal{E}}^{\alpha}(\Omega_T)$. More precisely, for any $1<\gamma<\beta:=(\alpha+\delta)/\alpha$ and $\gamma\leq 2$, we have
	\begin{displaymath}
		\|\xi\|_{\alpha,\mathcal{E}}^{\alpha}\leq \gamma^*\left\{\|\xi\|_{L_G^{\alpha+\delta}}^{\alpha}+14^{1/\gamma}
		C_{\beta/\gamma}\|\xi\|_{L_G^{\alpha+\delta}}^{(\alpha+\delta)/\gamma}\right\},\quad \forall \xi\in L_{ip}(\Omega_T),
	\end{displaymath}
	where $C_{\beta/\gamma}=\sum_{i=1}^\infty i^{-\beta/\gamma}$ and $\gamma^*=\gamma/(\gamma-1)$.
\end{theorem}

\subsection{$G$-BMO martingales and its properties}

In this section, we recall some results of $G$-BMO martingale and $G$-Girsanov theorem. More details can be found in \cite{CT,HLSH}.

\begin{definition}
    For $Z\in H^2_G(0,T)$, a symmetric $G$-martingale $\int_0^\cdot Z_sdB_s$ on $[0,T]$ is called a $G$-BMO martingale if
    \begin{align*}
       \|Z\|_{BMO_G}^2:=\sup_{\P\in\mathcal{P}}\|Z\|^2_{BMO(\P)}=\sup_{\P\in\mathcal{P}}\sup_{\tau\in\mathcal{T}_0^T}\Bigg\|\E^\P_\tau\left[\int_\tau^T|Z_t|^2d\langle B\rangle_t\right] \Bigg\|_{L^\infty(\P)}<\infty,
    \end{align*}
    where $\mathcal{T}_0^T$ denotes the collection of all $\mathbb{F}$-stopping times taking values in $[0,T]$ and $\|Z\|_{BMO(\P)}$ is the BMO norm of $\int_0^\cdot Z_s dB_s$ under probability $\P$. $Z$ is called a $G$-BMO martingale generator.
\end{definition}

We denote by $BMO_G$ be the collection of all $Z\in H^2_G(0,T)$ such that $\|Z\|_{BMO_G}<\infty$. 

\begin{lemma}\label{lem2.2}
    Given $Z\in BMO_G$, for any $\alpha\geq 1$ and any $t\in[0,T]$, we have 
    \begin{align*}
        \hE_t\left[\left(\int_t^T|Z_s|^2d\langle B\rangle_s\right)^{\frac{\alpha}{2}}\right]\leq C_\alpha\|Z\|_{BMO_G}^\alpha,
    \end{align*}
    where $C_\alpha>0$ is a constant depending only on $\alpha$.
\end{lemma}

\begin{lemma}\label{lem2.3}
    For $Z\in BMO_G$, the process
    \begin{align*}
        \mathcal{E}(Z)_t:=\exp\left(\int_0^t Z_s dB_s-\frac{1}{2}\int_0^t|Z_s|^2d\langle B\rangle_s\right), \ t\ge 0
    \end{align*}
    is a symmetric $G$-martingale.
\end{lemma}

\begin{lemma}\label{lem2.4 2.5}
    (i) Let $\phi(x)=\left(1+\frac{1}{x^2}\log \frac{2x-1}{2(x-1)}\right)^{\frac{1}{2}}-1$ and $1<q<\infty$. Suppose that $\|Z\|_{BMO_G}<\phi(q)$. Then, we have 
    \begin{align*}
        \sup_{\P\in\mathcal{P}}\sup_{\tau\in \mathcal{T}_0^T}\left\|\E^{\P}_\tau\left[\frac{\mathcal{E}(Z)_T^q}{\mathcal{E}(Z)_\tau^q}\right]\right\|_{L^\infty(\P)}\leq C_q,
    \end{align*}
    where $C_q>0$ is a constant depending only on $q$. 

    \noindent (ii) Given $1<r<\infty$, suppose that $\|Z\|_{BMO_G}<\frac{\sqrt{2}}{2}(\sqrt{r}-1)$. Then, we have
    \begin{align*}
        \sup_{\P\in\mathcal{P}}\sup_{\tau\in \mathcal{T}_0^T}\left\|\E^{\P}_\tau\left[\left(\frac{\mathcal{E}(Z)_\tau}{\mathcal{E}(Z)_T}\right)^{\frac{1}{r-1}}\right]\right\|_{L^\infty(\P)}\leq C_r,
    \end{align*}
    where $C_r>0$ is a constant depending only on $r$.
\end{lemma}

Given $Z\in BMO_G$ satisfying $\|Z\|_{BMO_G}<\phi(q)$, we may define a new $G$-(conditional) expectation $\widetilde{\E}$ associated with $\mathcal{E}(Z)$ as follows
\begin{equation}
    \widetilde{\E}_t[X]:=\hE_t\left[\frac{\mathcal{E}(Z)_T}{\mathcal{E}(Z)_t} X\right], \ X\in L^p_G(\Omega_T),
\end{equation}
where $p>\frac{q}{q-1}$. Then, we have the following Girsanov theorem. 

\begin{lemma}\label{lem2.6 2.7}
(i) Suppose that $Z\in BMO_G$. Then, the process $B-\int Zd\langle B\rangle$ is a $G$-Brownian motion under $\widetilde{\E}$.

 \noindent (ii) Suppose that $Z\in BMO_G$ satisfies $\|Z\|_{BMO_G}<\phi(q)$. Let $K$ be a non-increasing $G$-martingale with $K_0=0$, such that $K_t\in L^p_G(\Omega_t)$, $t\in[0,T]$ for some $p>\frac{q}{q-1}$. Then, $K$ is a non-increasing $G$-martingale under $\widetilde{\E}$.    
\end{lemma}

\subsection{One-dimensional quadratic $G$-BSDEs}

In this section, we now recall some basic results about $G$-BSDEs with quadratic generator (see \cite{CT,HLSH,HTW}). Consider the following type of $G$-BSDE
\begin{equation}\label{eq1.1}
	Y_t=\xi+\int_t^T f(s,Y_s,Z_s)ds+\int_t^T g(s,Y_s,Z_s)d\langle B\rangle_s-\int_t^T Z_s dB_s-(K_T-K_t),
\end{equation}
where the generators $f(t,\omega,y,z),  g(t,\omega,y,z):[0,T]\times\Omega_T\times\mathbb{R}\times\mathbb{R}\rightarrow \mathbb{R}$ and the terminal value $\xi$ 
satisfy the following properties:
\begin{description}
	\item[(H1)] $|f(t,\omega,0,0)|+|g(t,\omega,0,0)|+|\xi(\omega)|\leq M_0$, q.s.;
    \item[(H2)] The generators $f,g$ are uniformly continuous in $(t,\omega)$, that is, there exists a non-decreasing continuous function $\rho:\mathbb{R}^+\rightarrow \mathbb{R}^+$ with $\rho(0)=0$, such that 
    \begin{align*}
        &\sup_{y,z\in\mathbb{R}}|f(t,\omega,y,z)-f(t',\omega',y,z)|\leq \rho(|t-t'|+\|\omega-\omega'\|_\infty);\\
        &\sup_{y,z\in\mathbb{R}}|g(t,\omega,y,z)-g(t',\omega',y,z)|\leq \rho(|t-t'|+\|\omega-\omega'\|_\infty);
    \end{align*}
	\item[(H3)] There exist two positive constants $L_y, L_z$ such that for any $(t,\omega)\in[0,T]\times\Omega$,
	\begin{align*}
	    |f(t,\omega,y,z)-f(t,\omega,y',z')|&+|g(t,\omega,y,z)-g(t,\omega,y',z')|\\
        &\leq L_y|y-y'|+L_z(1+|z|+|z'|)|z-z'|.
	\end{align*}		
\end{description}

For any constant $\alpha\geq 1$, we denoted by $\mathfrak{S}_G^{\alpha}(0,T)$ the collection of processes $(Y,Z,K)$ such that $Y\in S_G^{\alpha}(0,T)$, $Z\in H_G^{\alpha}(0,T)$, and $K$ is a non-increasing $G$-martingale with $K_0=0$ and $K_T\in L_G^{\alpha}(\Omega_T)$. Hu, Lin and Soumana Hima \cite{HLSH} established the existence and uniqueness result for Equation \eqref{eq1.1}.

\begin{theorem}[\cite{HLSH}]\label{wellposedness for GBSDE}
	Assume that $\xi\in L_G^{\infty}(\Omega_T)$ and $f,g,\xi$ satisfy (H1)-(H3). Then, $G$-BSDE \eqref{eq1.1} has a unique solution $(Y,Z,K)\in \mathfrak{S}_G^{2}(0,T)$. Moreover, we have
	\begin{displaymath}
		\|Y\|_{S^\infty_G}+\|Z\|_{BMO_G}\leq C(M_0,L_y,L_z)
	\end{displaymath}
    and for any $p\geq 1$,
    \begin{align*}
        \hat{\mathbb{E}}[|K_T|^p]\leq C(p,M_0,L_y,L_z).
    \end{align*}
\end{theorem}

Similar with the Lipschitz case, the comparison theorem for  quadratic $G$-BSDEs still holds.
\begin{theorem}[\cite{CT}]\label{comparison theorem for GBSDE}
	Let the triplet $(\xi^l,f^l,g^l)$, $l=1,2$, satisfy (H1)-(H3). Let $(Y^l,Z^l,K^l)\in\mathfrak{S}^2_G(0,T)$, $l=1,2$, be the solution to the following $G$-BSDE:
	\begin{displaymath}
		Y^l_t=\xi^l+\int_t^T f^l(s,Y^l_s,Z^l_s)ds+\int_t^T g^l(s,Y^l_s,Z^l_s)d\langle B\rangle_s+V_T^l-V_t^l-\int_t^T Z^l_s dB_s-(K^l_T-K^l_t),
	\end{displaymath}
	where process $\{V_t^l\}_{0\leq t\leq T}$ is continuous with finite variation. Assume that 
    \begin{align*}
        (Y^l,Z^l,K^l_T,V^l)\in S^\infty_G(0,T)\times BMO_G\times\bigcap_{p\geq 1}L^p_G(\Omega_T)\times \bigcap_{p\geq 1}S^p_G(0,T),
    \end{align*}
    and $K^l$ is a non-increasing $G$-martingale. 
    If $\xi^1\geq \xi^2$, $f^1\geq f^2$, $g^1\geq g^2$, q.s. and $V^1-V^2$ is a non-decreasing process, then we have $Y_t^1\geq Y_t^2$, q.s. for any $t\in[0,T]$.
\end{theorem}

\begin{remark}
    Actually, as claimed in \cite{CT}, Theorem \ref{wellposedness for GBSDE} and Theorem \ref{comparison theorem for GBSDE} still hold if (H1) is replaced by the following condition
    \begin{itemize}
        \item[(H1')] $\int_0^T|f(t,\omega,0,0)|^2dt+\int_0^T|g(t,\omega,0,0)|^2dt+|\xi(\omega)|\leq M_0$, q.s.
    \end{itemize}
\end{remark}




\subsection{Quadratic reflected $G$-BSDEs with a lower obstacle}

Now we introduce the quadratic reflected $G$-BSDEs with a lower obstacle studied in \cite{CT}. Compared with the $G$-BSDEs, the parameters consist of a terminal value $\xi$, generators $f,g$ and an obstacle $S$, where $S$ satisfies the following assumption. 
\begin{description}
	\item[(H4)] $S\in \cap_{p>1}S^p_G(0,T)$. Moreover, there exists a positive constant $N_0$ such that $S_t\leq N_0$, q.s. for any $t\in[0,T]$.
    \item[(H5)] $S$ is uniformly continuous in $(t,\omega)$, that is, there exists a non-decreasing continuous function $\rho:\mathbb{R}^+\rightarrow \mathbb{R}^+$ with $\rho(0)=0$, such that 
    \begin{align*}
        |S_t(\omega)-S_{t'}(\omega')|\leq \rho(|t-t'|+\|\omega-\omega'\|_\infty).
    \end{align*}  
\end{description}

Let us now introduce the reflected $G$-BSDE with a lower obstacle. A triple of processes $(Y,Z,A)$ is called a solution of reflected $G$-BSDE with a lower obstacle with parameters $(\xi,f,g,S)$ if:
\begin{description}
	\item[(a)]$(Y,Z,A)\in\mathcal{S}_G^{\alpha}(0,T)$ and $Y_t\geq S_t$, $0\leq t\leq T$;
	\item[(b)]$Y_t=\xi+\int_t^T f(s,Y_s,Z_s)ds+\int_t^T g(s,Y_s,Z_s)d\langle B\rangle_s
	-\int_t^T Z_s dB_s+(A_T-A_t)$;
	\item[(c)]$\{-\int_0^t (Y_s-S_s)dA_s\}_{t\in[0,T]}$ is a non-increasing $G$-martingale.
\end{description}
Here, $\mathcal{S}_G^{\alpha}(0,T)$ is the collection of processes $(Y,Z,A)$ such that $Y\in S_G^{\alpha}(0,T)$, $Z\in H_G^{\alpha}(0,T)$, $A$ is a continuous non-decreasing process with $A_0=0$ and $A\in S_G^\alpha(0,T)$. By the results in \cite{CT}, we have the following existence and uniqueness result as well as the comparison theorem  for quadratic reflected $G$-BSDEs.

\begin{theorem}[\cite{CT}]\label{wellposedness for RGBSDE}
	Let the quadruple $(\xi,f,g,S)$ satisfy (H1)-(H5)  with $S_T\leq \xi$, q.s. Then, the reflected $G$-BSDE with parameters $(\xi,f,g,S)$ has a unique solution $(Y,Z,A)$ such that $Y\in S^\infty_G(0,T)$, $Z\in BMO_G$ and $A\in \cap_{p\geq 2}S^p_G(0,T)$.
\end{theorem}

\begin{theorem}[\cite{CT}]\label{comparison theorem for RGBSDE}
	Let $(\xi^i,f^i,g^i,S^i)$  be two sets of parameters satisfying (H1)-(H5)  with $S^i_T\leq \xi^i$, q.s., $i=1,2$. Let  $(Y^i,Z^i,A^i)\in\mathcal{S}^2_G(0,T)$ be the solution to the reflected $G$-BSDE with parameters $(\xi^i,f^i,g^i,S^i)$,  $i=1,2$. Assume that $Y^i\in S^\infty(0,T)$, $Z^i\in BMO_G$ and $A^i_T\in \cap_{p\geq 2}L^p_G(\Omega_T)$ for $i=1,2$. If $\xi^1\geq \xi^2$, $f^1\geq f^2$, $g^1\geq g^2$ and $S^1\geq S^2$, q.s., then
	\begin{displaymath}
		Y_t^1\leq Y^2_t,  \quad \textrm{q.s.}, \quad 0\leq t\leq T.
	\end{displaymath}
\end{theorem}

\section{Doubly reflected $G$-BSDE with quadratic generator} \label{sec:Doubly reflected $G$-BSDE with quadratic generator}

Now we give the definition of solutions to doubly reflected $G$-BSDEs when the generators has quadratic growth in $z$. A triple of processes $(Y,Z,A)$ with $Y\in S_G^\infty(0,T)$, $Z\in BMO_G$ is called a solution to the doubly reflected $G$-BSDE with the parameters $(\xi, f, g, L, U)$  if the following properties hold:
\begin{description}
	\item[(S1)] $L_t\leq Y_t\leq U_t$, $t\in[0,T]$;
	\item[(S2)] $Y_t=\xi+\int_t^T f(s,Y_s,Z_s)ds+\int_t^T g(s,Y_s,Z_s)d\langle B\rangle_s-\int_t^T Z_s dB_s+(A_T-A_t)$;	
	\item[(S3)] $(Y, A)$ satisfies  Approximate Skorohod Condition (AMC for short).
\end{description}

\noindent \textbf{Condition} (AMC): We say a pair of processes $(Y,  A)\in S^\infty_G(0,T)\times \cap_{\alpha\geq 2}S^\alpha_G(0,T)$  satisfies the approximate Skorohod condition (w.r.t. the obstacles $L, U$) if  there exist non-decreasing processes $\{A^{n,+}\}_{n\in\mathbb{N}}$, $\{A^{n,-}\}_{n\in\mathbb{N}}$ and  non-increasing $G$-martingales $\{K^n\}_{n\in\mathbb{N}}$, such that for any $\alpha\geq 2$
\begin{itemize}
	\item $\hat{\mathbb{E}}\left[|A_T^{n,+}|^\alpha+|K^n_T|^\alpha\right]\leq C_\alpha$ and $\|A^{n,-}\|_{L^\infty_G}\leq C$, where $C_\alpha$  and $C$ are independent of $n$;
	\item $\hat{\mathbb{E}}\left[\sup\limits_{t\in[0,T]}|A_t-(A_t^{n,+}-A_t^{n,-}-K_t^n)|^\alpha\right]\rightarrow 0$, as $n\rightarrow\infty$;
	\item  $\lim\limits_{n\rightarrow\infty}\hat{\mathbb{E}}\left[\left|\int_0^T (Y_s-L_s)d A_s^{n,+}\right|^{\alpha/2}\right]=0$;
	\item  $\lim\limits_{n\rightarrow\infty}\hat{\mathbb{E}}\left[\left|\int_0^T (U_s-Y_s)d A_s^{n,-}\right|^{\alpha/2}\right]=0$.
\end{itemize}
We call $\{A^{n,+}\}_{n\in\mathbb{N}}$, $\{A^{n,-}\}_{n\in\mathbb{N}}$ and $\{K^n\}_{n\in\mathbb{N}}$ the approximate sequences for $(Y,A)$  w.r.t. the obstacles $L,U$.

\begin{assumption}\label{assLU}
\begin{itemize}
    \item[(1)] $L,U\in \cap_{p>1}S^p_G(0,T)$, $L_T,U_T\in L^\infty_G(\Omega_T)$ and $L_T\leq \xi\leq U_T$. Moreover, there exists a constant $N_0$ such that $L_t\leq N_0$ and $-U_t\leq N_0$  q.s. for any $t\in[0,T]$.
    \item[(2)] $L,U$ are uniformly continuous in $(t,\omega)$, that is, there exists a non-decreasing continuous function $\rho:\mathbb{R}^+\rightarrow \mathbb{R}^+$ with $\rho(0)=0$, such that 
    \begin{align*}
        |L_t(\omega)-L_{t'}(\omega')|+|U_t(\omega)-U_{t'}(\omega')|\leq \rho(|t-t'|+\|\omega-\omega'\|_\infty).
    \end{align*}  
    \item[(3)] There exists some $I\in S_G^\infty(0,T)$ satisfying the following representation
	\begin{displaymath}
		I_t=I_0+A^{I,-}_t-A_t^{I,+}+\int_0^t \sigma^I(s)dB_s,
	\end{displaymath}
	where $\sigma^I\in S^\infty_G(0,T)$, $A^{I,+}$, $A^{I,-}\in \cap_{p\geq1}S^p_G(0,T)$ are uniformly continuous in $(t,\omega)$ and $A^{I,+}$, $A^{I,-}$ are two non-decreasing processes with $A^{I,+}_0=A^{I,-}_0=0$ and $A^{I,+}_T,A^{I,-}_T\in L^\infty_G(\Omega_T)$ such that $L_t\leq I_t\leq U_t$.
    \item[(4)] $U+A^{I,+}$ has the following representation 
	\begin{equation}\label{A3}
		U_t+A^{I,+}_t=U_0+\int_0^t b(s)ds+\int_0^t \sigma(s)dB_s+K^u_t, %
	\end{equation}
	where $b,\sigma\in S_G^\infty(0,T)$, and $K^u\in \cap _{p\geq 1}S_G^p(0,T)$ is a non-increasing $G$-martingale with $K^u_0=0$.
\end{itemize}    
\end{assumption}

The following is the main result in this section.

\begin{theorem}\label{main1}
	Suppose that $\xi$, $f$, $g$ satisfy (H1)-(H3) and $L$, $U$ satisfy Assumption \ref{assLU}. Then the doubly reflected $G$-BSDE with data $(\xi,f,g,L,U)$ has a unique solution $(Y,Z,A)\in S^\infty_G(0,T)\times BMO_G\times \cap_{\alpha\geq 2} S^\alpha_G(0,T)$. Moreover, for any $\alpha\ge 2$, we have
	\begin{align*}
		\lim_{n\rightarrow \infty}\hat{\mathbb{E}}\left[\sup_{t\in[0,T]}|Y_t-\bar{Y}^n_t|^\alpha\right]=0, 
	\end{align*}
	where for each $n\in\mathbb{N}$, $(\bar{Y}^n,\bar{Z}^n,\bar{A}^n)\in S^\infty_G(0,T)\times BMO_G\times \cap_{\alpha\geq 2} S^\alpha_G(0,T)$ is the solution to the following reflected $G$-BSDE:
	\begin{equation}\label{penalized RGBSDE}
		\begin{cases}
			\bar{Y}^n_t=\xi+\int_t^T f(s,\bar{Y}^n_s,\bar{Z}^n_s)ds+\int_t^T g(s,\bar{Y}^n_s,\bar{Z}^n_s)d\langle B\rangle_s\\
			\ \ \ \ \ \ \ \ -n\int_t^T(\bar{Y}^n_s-U_s)^+ds-\int_t^T \bar{Z}^n_sdB_s+(\bar{A}^n_T-\bar{A}^n_t), \\
			\bar{Y}^n_t\geq L_t, \forall t\in[0,T], \{-\int_0^t (\bar{Y}^n_s-L_s)d\bar{A}^n_s\}_{t\in[0,T]} \textrm{ is a non-increasing $G$-martingale}.
		\end{cases}
	\end{equation}
\end{theorem}

\subsection{A priori estimates}

In this subsection, we will provide some a priori estimates for solutions to doubly reflected $G$-BSDEs. The following proposition can be regarded as an extension of Proposition 3.1 in \cite{CT} and Theorem 2.4 in \cite{LS}.

\begin{proposition}\label{est-Z-A}
    Let $\xi,g$ satisfy (H1) and (H3). Assume that $(Y,Z,A,A')$ solves
    \begin{align*}
        Y_t=\xi+\int_t^T g(s,Y_s,Z_s)d\langle B\rangle_s-\int_t^T Z_sdB_s+(A_T-A_t)-(A'_T-A'_t),
    \end{align*}
    where 
    \begin{align*}
        (Y,Z)\in S_G^\infty(0,T)\times H^2_G(0,T),
    \end{align*}
    and $A$, $A'$ are continuous nondecreasing processes with $A_0=A'_0=0$.

 \begin{itemize}
     \item[(1)] Suppose that $A'_T\in L^\infty_G(\Omega_T)$.  Then, for any $\alpha\geq 1$, there exist two constants 
     \begin{align*}
         &C_1:=C_1(\|Y\|_{S^\infty_G},\|A'_T\|_{L^\infty_G}, T,L_y,L_z,M_0,\bar{\sigma}),\\
         &C_2:=C_2(\|Y\|_{S^\infty_G},\|A'_T\|_{L^\infty_G}, T,L_y,L_z,M_0,\bar{\sigma},\alpha),
     \end{align*}
     such that 
    \begin{align*}
        \|Z\|_{BMO_G}\leq C_1,\  \hat{\E}[|A_T|^\alpha]\leq C_2.
    \end{align*}
    \item[(2)] Suppose that $A_T\in L^\infty_G(\Omega_T)$.  Then, for any $\alpha\geq 1$, there exist two constants 
    \begin{align*}
         &C'_1:=C'_1(\|Y\|_{S^\infty_G},\|A_T\|_{L^\infty_G}, T,L_y,L_z,M_0,\bar{\sigma}),\\
         &C'_2:=C'_2(\|Y\|_{S^\infty_G},\|A_T\|_{L^\infty_G}, T,L_y,L_z,M_0,\bar{\sigma},\alpha),
     \end{align*}
     such that 
    \begin{align*}
        \|Z\|_{BMO_G}\leq C'_1,\  \hat{\E}[|A'_T|^\alpha]\leq C'_2.
    \end{align*}
 \end{itemize}

\end{proposition}

\begin{proof}
    The proof is similar with the one for Proposition 3.1 in \cite{CT}. For readers' convenience, we give a short proof for the first case here. Note that for any $y,z$, we have
    \begin{align*}
        |g(s,y,z)|\leq |g(s,0,0)|+L_y|y|+L_z(|z|+|z|^2)\leq |g(s,0,0)|+L_y|y|+\frac{1}{2}L_z+\frac{3}{2}L_z|z|^2.
    \end{align*}
    For some $a>0$, which will be determined later, applying It\^{o}'s formula to $e^{-aY_t}$ under each $\P\in\mathcal{P}$, for each $\tau\in\mathcal{T}_0^T$, we have
    \begin{align*}
        &\frac{a^2}{2}\int_\tau^T e^{-aY_s}Z_s^2d\langle B\rangle_s\\
        =&e^{-a\xi}-e^{-aY_\tau}-\int_\tau^T a e^{-aY_s}g(s,Y_s,Z_s)d\langle B\rangle_s+\int_\tau^T a e^{-aY_s}Z_s dB_s\\
        &-\int_\tau^T a e^{-aY_s}dA_s+\int_\tau^T a e^{-aY_s}dA'_s\\
        \leq &e^{-a\xi}-e^{-aY_\tau}+\int_\tau^T a e^{-aY_s}(|g(s,0,0)|+\frac{1}{2}L_z+L_y|Y_s|)d\langle B\rangle_s\\
        &+\frac{3a L_z}{2}\int_\tau^T e^{-aY_s}Z_s^2d\langle B\rangle_s+\int_\tau^T a e^{-aY_s}Z_s dB_s+\int_\tau^T a e^{-aY_s}dA'_s.
    \end{align*}
    Choosing $a=4L_z$ and taking conditional expectations under $\P$ on both sides yield that 
    \begin{align*}
        &2L^2_z\E^\P_\tau\left[\int_\tau^T e^{-aY_s}Z_s^2d\langle B\rangle_s\right]\\
        \leq &\E^\P_\tau\left[e^{-a\xi}-e^{-aY_\tau}+\int_\tau^T a e^{-aY_s}(|g(s,0,0)|+\frac{1}{2}L_z+L_y|Y_s|)d\langle B\rangle_s+\int_\tau^T a e^{-aY_s}dA'_s\right]\\
        \leq &2 e^{4L_z\|Y\|_{S^\infty_G}}+4L_z\bar{\sigma}^2\left(M_0 T+\frac{1}{2}L_z T+L_y T \|Y\|_{S^\infty_G}\right)e^{4L_z\|Y\|_{S^\infty_G}}+4L_ze^{4L_z\|Y\|_{S^\infty_G}}\|A'_T\|_{L^\infty_G}.
    \end{align*}
    Since $\tau\in \mathcal{T}_0^T$ and $\P\in\mathcal{P}$ are arbitrarily chosen, we obtain the estimate for $\|Z\|_{BMO_G}$. 

    For the estimate of $A$, note that 
    \begin{align*}
        A_T=Y_0-\xi-\int_0^T g(s,Y_s,Z_s)d\langle B\rangle_s+\int_0^T Z_s dB_s+A'_T.
    \end{align*}
    Comparing with the estimate for $\hat{\E}[|A_T|^\alpha]$ in Proposition 3.1 in \cite{CT}, for the estimate in the present setting, there will be an additional term $C_\alpha\hat{\E}[|A'_T|^\alpha]$, where $C_\alpha$ is a constant only depending on $\alpha$. Then, we obtain the desired result.
\end{proof}

The following result is an extension of Proposition 3.3 in \cite{CT} and Proposition 3.7 in \cite{LN}. 

\begin{proposition}\label{est-Y1-Y2}
    Let $(\xi^i,g^i,L^i,U^i)$ be two sets of parameters such that $(\xi^i,g^i)$ satisfies (H1) and (H3), $L^i,U^i$ satisfy Assumption \ref{assLU} (1), $i=1,2$. Suppose that $(Y^i,Z^i,A^i)\in S_G^\infty(0,T)\times BMO_G\times\cap_{\alpha\geq 2}S^\alpha_G(0,T)$ is a solution to the doubly reflected $G$-BSDE with parameters $(\xi^i,g^i,L^i,U^i)$, $i=1,2$. Moreover, assume that 
    \begin{align*}
        \|L_z(1+|Z^1|+|Z^2|)\|_{BMO_G}<\phi(q):=\left(1+\frac{1}{q^2}\log \frac{2q-1}{2(q-1)}\right)^{\frac{1}{2}}-1,
    \end{align*}
    where $q>1$. Then, for any $p>\frac{q}{q-1}$, there exists a constant $C:=C(p,T,L_z,L_y,\bar{\sigma},N_0)$, such that for any $t\in[0,T]$,
    \begin{align*}
        |Y^1_t-Y^2_t|^2\leq &C\left(\hE_t\left[|\xi^1-\xi^2|^{2p}\right]\right)^{\frac{1}{p}}+C\left(\hE_t\left[\left(\int_t^T|\hat{\lambda}_s|^2 d\langle B \rangle_s\right)^p\right]\right)^{\frac{1}{p}}\\
        &+C\Psi_{t,T}\left(\hat{\mathbb{E}}_t\left[\sup_{s\in[t,T]}|\hat{L}_s|^{2p}\right]+\hat{\mathbb{E}}_t\left[\sup_{s\in[t,T]}|\hat{U}_s|^{2p}\right]\right)^{\frac{1}{2p}},
    \end{align*}
    where $\hat{\lambda}_s=g^1(s,Y^2_s,Z^2_s)-g^2(s,Y^2_s,Z^2_s)$ and
    \begin{align*}
        \Psi_{t,T}=\liminf_{n\rightarrow\infty}\left(\sum_{i=1}^2\left(\hE_t\left[|A^{i,n,+}_T|^{2p}\right]+\hE_t\left[|A^{i,n,-}_T|^{2p}\right]\right)\right)^{\frac{1}{2p}}.
    \end{align*}
\end{proposition}

\begin{proof}
We define $\hat{\xi}:=\xi^1-\xi^2$ and
\begin{align*}
    \hat{Y}_t:=Y^1_t-Y^2_t, \ \hat{Z}_t:=Z^1_t-Z^2_t,\ \hat{L}_t:=L^1_t-L^2_t, \ \hat{U}_t:=U^1_t-U^2_t, \ \hat{A}_t:=A^1_t-A^2_t.
\end{align*}
For each given $\varepsilon>0$, let $l$ be a Lipschitz continuous function such that $I_{[-\varepsilon,\varepsilon]}(x)\leq l(x)\leq I_{[-2\varepsilon,2\varepsilon]}(x)$ for $x\in\mathbb{R}$. For any $s\in[0,T]$, set 
    \begin{align*}
        \hat{a}^{\varepsilon}_s&:=[1-l(\hat{Y}_s )]\frac{{g}^1(s,{Y}^1_s,{Z}^1_s)-{g}^1(s,Y^2_s,{Z}^1_s)}{{\hat{Y}_s}}I_{\{|\hat{Y}_s|>0\}},\\
        \hat{b}^{\varepsilon}_s&:=[1-l(\hat{Z}_s)]\frac{{g}^1(s,Y^2_s,{Z}^1_s)-{g}^1(s,Y^2_s,Z^2_s)}{|\hat{Z}_s|^2}\hat{Z}_sI_{\{|\hat{Z}_s|>0\}},\\
        m^{\varepsilon}_s&:=l(\hat{Y}_s)[{g}^1(s,{Y}^1_s,{Z}^1_s)-{g}^1(s,Y^2_s,{Z}^1_s)]+l(\hat{Z}_s)[{g}^1(s,Y^2_s,{Z}^1_s)-{g}^1(s,Y^2_s,Z^2_s)].
    \end{align*}
    It is easy to check that 
    \begin{align*}
        g^1(s,Y^1_s,Z^1_s)-g^2(s,Y^2_s,Z^2_s)=\hat{\lambda}_s+\hat{m}^\varepsilon_s+\hat{a}^\varepsilon_s \hat{Y}_s+\hat{b}^\varepsilon_s \hat{Z}_s
    \end{align*}
    and
    \begin{align*}
        |\hat{a}^{\varepsilon}_s|\leq L_y, \ |\hat{b}^{\varepsilon}_s|\leq L_z(1+|{Z}^1_s|+|Z^2_s|), \ |\hat{m}^{\varepsilon}_s|\leq 2\varepsilon(L_y+L_z(1+2\varepsilon+2|Z^1_s|)).
    \end{align*}
    Moreover, by Lemma 3.6 in \cite{HLSH}, $\hat{b}^{\varepsilon}\in H^2_G(0,T)$ is a $G$-BMO martingale generator. We may define a new conditional $G$-expectation $\widetilde{\mathbb{E}}^{\varepsilon}_t$ as follows
    \begin{align*}
        \widetilde{\mathbb{E}}^{\varepsilon}_t[X]:=\hat{\E}_t\left[\frac{\mathcal{E}(\hat{b}^{\varepsilon})_T}{\mathcal{E}(\hat{b}^{\varepsilon})_t}X\right], \ X\in L^p(\Omega_T).
    \end{align*}
    By Lemma \ref{lem2.6 2.7}, $\widetilde{B}^{\varepsilon}$ is a $G$-Brownian motion under $\widetilde{\E}^{\varepsilon}$, where $\widetilde{B}^{\varepsilon}_t=B_t-\int_0^t \hat{b}^{\varepsilon}_sd\langle B\rangle_s$. Set $p'=\frac{p}{p-1}$. Note the fact that 
    \begin{align*}
        \|\hat{b}^\varepsilon\|_{BMO_G}<\phi(q)<\phi(p').
    \end{align*}
    By Lemma \ref{lem2.4 2.5}, for any $X\in L^{p}_G(\Omega_T)$ we have
    \begin{align}\label{est-widetildeE}
        \widetilde{\E}^\varepsilon_t[X]\leq \left(\hat{\E}_t\left[\frac{\mathcal{E}(b^{\varepsilon})^{p'}_T}{\mathcal{E}(b^{\varepsilon})^{p'}_t}\right]\right)^{\frac{1}{p'}}\left(\hat{\E}_t[|X|^{p}]\right)^{\frac{1}{p}}\leq C_p \left(\hat{\E}_t[|X|^{p}]\right)^{\frac{1}{p}}, \ \textrm{q.s.}
    \end{align}
     Then, we have
    \begin{align*}
        \hat{Y}_t=\hat{\xi}+\int_t^T \left(\hat{\lambda}_s+\hat{m}^\varepsilon_s+\hat{a}^\varepsilon_s\hat{Y}_s\right)d\langle B\rangle_s-\int_t^T\hat{Z}_sd\widetilde{B}^\varepsilon_s+\hat{A}_T-\hat{A}_t.
    \end{align*}
    For any fixed $r>0$, applying It\^{o}'s formula to $e^{rt}|\hat{Y}_t|^2$, for any $t\in[0,T]$, we have
    \begin{align*}
        &e^{rt}|\hat{Y}_t|^2+r\int_t^T e^{rs}|\hat{Y}_s|^2ds+\int_t^T e^{rs}|\hat{Z}_s|^2d\langle B\rangle_s\\
        =&e^{rT}|\hat{\xi}|^2+\int_t^T 2e^{rs}\hat{Y}_s\left(\hat{\lambda}_s+\hat{m}^\varepsilon_s+\hat{a}^\varepsilon_s\hat{Y}_s\right)d\langle B\rangle_s-\int_t^T 2e^{rs}\hat{Y}_s \hat{Z}_sd\widetilde{B}^\varepsilon_s+\int_t^T2 e^{rs}\hat{Y}_sd\hat{A}_s\\
        \leq &e^{rT}|\hat{\xi}|^2+\int_t^T e^{rs}\left(|\hat{\lambda}_s|^2+|\hat{m}^\varepsilon_s|^2\right)d\langle B\rangle_s +2(1+L_y)\int_t^Te^{rs}|\hat{Y}_s|^2 d\langle B\rangle_s\\
        &-\int_t^T 2e^{rs}\hat{Y}_s \hat{Z}_sd\widetilde{B}^\varepsilon_s+\int_t^T2 e^{rs}\hat{Y}_sd\hat{A}_s.
    \end{align*}
    Set $A^{i,n}=A^{i,n,+}-A^{i,n,-}-K^{i,n}$, $i=1,2$, $\hat{Y}^L_t=(Y^1_t-L^1_t)-(Y^2_t-L_t^2)$ and $\hat{Y}^U_t=(U^1_t-Y_t^1)-(U^2_t-Y_t^2)$. Noting that $\hat{Y}^L_t\leq Y^1_t-L^1_t$, $\hat{Y}^U_t \leq U_t^1-Y_t^1$ and $A^{1,n,+}$, $A^{1,n,-}$, $-K^{1,n}$ are non-decreasing processes, it is easy to check that
	\begin{align*}
		\int_t^Te^{rs}\hat{Y}_sdA^{1}_s
		=&\int_t^T e^{rs}\hat{Y}_sd({A}^1_s-A_s^{1,n})+\int_t^Te^{rs}\hat{Y}_sd{A}^{1,n}_s\\
		\leq &\int_t^T e^{rs}(Y^1_s-L^1_s)dA_s^{1,n,+}+\int_t^T  e^{rs}|\hat{L}_s|dA_s^{1,n,+}\\
		&+\int_t^T e^{rs}(U_s^1-Y_s^1)dA_s^{1,n,-}+\int_t^T  e^{rs}|\hat{U}_s|dA_s^{1,n,-}\\
		&+\left|\int_t^T e^{rs}\hat{Y}_sd({A}^1_s-A_s^{1,n})\right|-\int_t^T e^{rs}(\hat{Y}_s)^+d{K}^{1,n}_s.
	\end{align*}
    Similarly, we have
    \begin{align*}
		-\int_t^Te^{rs}\hat{Y}_sdA^{2}_s
		\leq &\int_t^T e^{rs}(Y^2_s-L^2_s)dA_s^{2,n,+}+\int_t^T  e^{rs}|\hat{L}_s|dA_s^{2,n,+}\\
		&+\int_t^T e^{rs}(U_s^2-Y_s^2)dA_s^{2,n,-}+\int_t^T  e^{rs}|\hat{U}_s|dA_s^{2,n,-}\\
		&+\left|\int_t^T e^{rs}\hat{Y}_sd({A}^2_s-A_s^{2,n})\right|-\int_t^T e^{rs}(\hat{Y}_s)^-d{K}^{2,n}_s.
	\end{align*}
    Set $r>2\bar{\sigma}^2(1+L_y)$ and 
    \begin{align*}
        M^{n,\varepsilon}_t=\int_0^t 2e^{rs}\hat{Y}_s \hat{Z}_sd\widetilde{B}^\varepsilon_s+\int_0^t 2e^{rs}\hat{Y}^+_sdK^{1,n}_s+\int_0^t 2e^{rs}\hat{Y}^-_sdK^{2,n}_s.
    \end{align*}
    All the above analysis indicates that 
    \begin{equation}\label{erthatYt}\begin{split}
         e^{rt}|\hat{Y}_t|^2+M^{n,\varepsilon}_T-M^{n,\varepsilon}_t\leq &e^{rT}|\hat{\xi}|^2+\int_t^T e^{rs}\left(|\hat{\lambda}_s|^2+|\hat{m}^\varepsilon_s|^2\right)d\langle B\rangle_s\\
         &+\int_t^T e^{rs}(Y^1_s-L^1_s)dA_s^{1,n,+}+\int_t^T  e^{rs}|\hat{L}_s|dA_s^{1,n,+}\\
		&+\int_t^T e^{rs}(U_s^1-Y_s^1)dA_s^{1,n,-}+\int_t^T  e^{rs}|\hat{U}_s|dA_s^{1,n,-}\\
        &+\int_t^T e^{rs}(Y^2_s-L^2_s)dA_s^{2,n,+}+\int_t^T  e^{rs}|\hat{L}_s|dA_s^{2,n,+}\\
		&+\int_t^T e^{rs}(U_s^2-Y_s^2)dA_s^{2,n,-}+\int_t^T  e^{rs}|\hat{U}_s|dA_s^{2,n,-}\\
        &+\left|\int_t^T e^{rs}\hat{Y}_sd({A}^1_s-A_s^{1,n})\right|+\left|\int_t^T e^{rs}\hat{Y}_sd({A}^2_s-A_s^{2,n})\right|.
    \end{split}\end{equation}
It is easy to check that $\{\int_0^te^{rs}\hat{Y}^+_sdK^{1,n}_s\}_{t\in[0,T]}$ is a non-increasing $G$-martingale under $\hat{\E}$ and for any $\alpha\geq 2$
    \begin{align*}
        \hat{\E}\left[\left(\int_0^t \hat{Y}^+_sdK^{1,n}_s\right)^\alpha\right]\leq \left(\hat{\E}\left[\sup_{t\in[0,T]}|\hat{Y}_t|^{2\alpha}\right]\right)^{1/2}\left(\hat{\E}\left[|K^{1,n}_T|^{2\alpha}\right]\right)^{1/2}<\infty.
    \end{align*}
    Moreover, noting the fact that
    \begin{align*}
        \|\hat{b}^\varepsilon\|_{BMO_G}\leq \|L_z(1+|Z^1|+|Z^2|)\|_{BMO_G}<\phi(q),
    \end{align*}
    by Lemma \ref{lem2.6 2.7}, $\{\int_0^te^{rs}\hat{Y}^+_sdK^{1,n}_s\}_{t\in[0,T]}$ is a non-increasing $G$-martingale under $\widetilde{\E}^\varepsilon$. By a similar analysis as the proof of Lemma 3.4 in \cite{HJPS1}, the process $M^{n,\varepsilon}$ is a $G$-martingale under $\widetilde{\E}^\varepsilon$. 
 Since for any $\alpha\geq 2$, we have $\hat{Y}\in S^\alpha_G(0,T)$ and 
\begin{align*}
     \lim_{n\rightarrow\infty}\hat{\mathbb{E}}\left[\sup_{t\in[0,T]}|A^1_t-A^{1,n}_t|^\alpha\right]=0.
 \end{align*} 
 By a similar analysis as the proof of Lemma 3.1 in \cite{LS},  we obtain that 
    \begin{align*}
        \lim_{n\rightarrow\infty}\hE\left[\sup_{t\in[0,T]}\left|\int_0^t e^{rs}\hat{Y}_sd(A^1_s-A^{1,n}_s)\right|^{\frac{\alpha}{2}}\right]=0.
    \end{align*}
    Applying Equation \eqref{est-widetildeE} implies that 
    \begin{align*}
        \lim_{n\rightarrow\infty}\widetilde{\E}^\varepsilon_t\left[\left|\int_t^T e^{rs}\hat{Y}_sd({A}^1_s-A_s^{1,n})\right|\right]=0.
    \end{align*}
    Similarly, we have
    \begin{align*}
        \lim_{n\rightarrow\infty}\widetilde{\E}^\varepsilon_t\left[\left|\int_t^T e^{rs}\hat{Y}_sd({A}^2_s-A_s^{2,n})\right|\right]=0.
    \end{align*}
 Since $(Y^1,A^1)$ satisfies $(\textmd{ASC})$, then for any $\alpha\geq 2$, we have
    \begin{align*}
       \lim_{n\rightarrow\infty} \hE_t\left[\left|\int_t^T (Y^1_s-L^1_s)dA^{1,n,+}_s\right|^\alpha\right]=0.
    \end{align*}
    Applying Equation \eqref{est-widetildeE}  yields that
    \begin{align*}
       \lim_{n\rightarrow\infty} \widetilde{\E}^\varepsilon_t\left[\int_t^T (Y^1_s-L^1_s)dA^{1,n,+}_s\right]=0.
    \end{align*}
    Similarly, we have
    \begin{align*}
    &\lim_{n\rightarrow\infty} \widetilde{\E}^\varepsilon_t\left[\int_t^T (Y^2_s-L^2_s)dA^{2,n,+}_s\right]=0,\\
       &\lim_{n\rightarrow\infty} \widetilde{\E}^\varepsilon_t\left[\int_t^T (U^1_s-Y^1_s)dA^{1,n,-}_s\right]=0,\\
       &\lim_{n\rightarrow\infty} \widetilde{\E}^\varepsilon_t\left[\int_t^T (U^2_s-Y^2_s)dA^{2,n,-}_s\right]=0.
    \end{align*}
Applying \eqref{est-widetildeE} and the H\"{o}lder inequality, we obtain that
	\begin{align*}
		\widetilde{\mathbb{E}}^\varepsilon_t\left[\int_t^T |\hat{L}_s|dA_s^{1,n,+}\right]
		\leq& C_p\left(\hat{\mathbb{E}}_t\left[\sup_{s\in[t,T]}|\hat{L}_s|^p|A^{1,n,+}_T|^p\right]\right)^{\frac{1}{p}}\\
		\leq &C_p\left(\hat{\mathbb{E}}_t\left[\sup_{s\in[t,T]}|\hat{L}_s|^{2p}\right]\right)^{\frac{1}{2p}}\left(\hat{\mathbb{E}}_t\left[|A_T^{1,n,+}|^{2p}\right]\right)^{\frac{1}{2p}}.
	\end{align*}
	 Similarly, we have 	\begin{align*}
		&\widetilde{\mathbb{E}}^\varepsilon_t\left[\int_t^T |\hat{U}_s|dA_s^{1,n,-}\right]
		\leq C_p\left(\hat{\mathbb{E}}_t\left[\sup_{s\in[t,T]}|\hat{U}_s|^{2p}\right]\right)^{\frac{1}{2p}}\left(\hat{\mathbb{E}}_t\left[|A_T^{1,n,-}|^{2p}\right]\right)^{\frac{1}{2p}},\\
		&\widetilde{\mathbb{E}}^\varepsilon_t\left[\int_t^T |\hat{L}_s|dA_s^{2,n,+}\right]
		\leq C_p\left(\hat{\mathbb{E}}_t\left[\sup_{s\in[t,T]}|\hat{L}_s|^{2p}\right]\right)^{\frac{1}{2p}}\left(\hat{\mathbb{E}}_t\left[|A_T^{2,n,+}|^{2p}\right]\right)^{\frac{1}{2p}},\\
        &\widetilde{\mathbb{E}}^\varepsilon_t\left[\int_t^T |\hat{U}_s|dA_s^{2,n,-}\right]
		\leq C_p\left(\hat{\mathbb{E}}_t\left[\sup_{s\in[t,T]}|\hat{U}_s|^{2p}\right]\right)^{\frac{1}{2p}}\left(\hat{\mathbb{E}}_t\left[|A_T^{2,n,-}|^{2p}\right]\right)^{\frac{1}{2p}}.
	\end{align*}
By the proof of Proposition 3.3 in \cite{CT}, we have 
    \begin{align*}
        \lim_{\varepsilon\rightarrow 0}\widetilde{\E}^\varepsilon_t\left[\int_t^T |\hat{m}^\varepsilon_s|^2d\langle B\rangle_s\right]=0.
    \end{align*}
    Finally, taking conditional expectation $\widetilde{\E}^\varepsilon_t$ on both sides of \eqref{erthatYt}, letting $n\rightarrow \infty$ first and then letting $\varepsilon\rightarrow 0$, we obtain the desired result.
\end{proof}

\subsection{Approximation via penalization}

By the proof of Theorem \ref{wellposedness for RGBSDE} in \cite{CT},  for each fixed $n\in\mathbb{N}$, the solution to reflected $G$-BSDE  \eqref{penalized RGBSDE} can be approximated by the solutions to the following family of $G$-BSDEs parameterized by $m\in\mathbb{N}$:
\begin{equation}\label{Ynm}\begin{split}
		Y^{n,m}_t=&\xi+\int_t^T f(s,Y^{n,m}_s,Z^{n,m}_s)ds+\int_t^T g(s,Y^{n,m}_s,Z^{n,m}_s)d\langle B\rangle_s-\int_t^T Z_s^{n,m}dB_s\\
        &-(K_T^{n,m}-K_t^{n,m})
		+\int_t^T m(Y_s^{n,m}-L_s)^-ds-\int_t^T n(Y_s^{n,m}-U_s)^+ds.
\end{split}\end{equation}
Set $A^{n,m,+}_t=\int_0^t m(Y_s^{n,m}-L_s)^-ds$ and $A^{n,m,-}_t=\int_0^t n(Y_s^{n,m}-U_s)^+ds$. Clearly, $A^{n,m,+}$ and $A^{n,m,-}$ are non-decreasing processes and Equation \eqref{Ynm} can be written as:
\begin{equation}
	\begin{split}
		Y^{n,m}_t=&\xi+\int_t^T f(s,Y^{n,m}_s,Z^{n,m}_s)ds+\int_t^T g(s,Y^{n,m}_s,Z^{n,m}_s)d\langle B\rangle_s-\int_t^T Z_s^{n,m}dB_s\\
        &-(K_T^{n,m}-K_t^{n,m}) +(A^{n,m,+}_T-A^{n,m,+}_t)-(A^{n,m,-}_T-A^{n,m,-}_t).
\end{split}\end{equation}
Without loss of generality, in the following of this subsection, we assume that $f\equiv 0$. Moreover, in this subsection, we assume that all the conditions in Theorem \ref{main1} hold. We first establish the uniform estimate for $\{Y^{n,m}\}_{n,m=1}^\infty$, which is an extension of Lemma 3.3 in \cite{LN}.
\begin{lemma}\label{est-Ynm}
	We have
	\begin{displaymath}
		\|Y^{n,m}\|_{S^\infty_G}\leq C(M_0,N_0,L_y,L_z,T).
	\end{displaymath}
\end{lemma}

\begin{proof}
	Without loss of generality, we assume that $|A^{I,+}_T|\leq N_0$ and $|A^{I,-}_T|\leq N_0$. Set $Y^*_t=I_t$, $Z^*_t=\sigma^I_t$. It is easy to check that
	\begin{equation}\label{Y^*}\begin{split}
			Y^*_t&=I_T-\int_t^T Z^*_sd B_s+(A^{I,+}_T-A^{I,+}_t)-(A^{I,-}_T-A^{I,-}_t)\\
			&=I_T+\int_t^T g(s,Y^*_s,Z^*_s)d\langle B\rangle_s-\int_t^T Z^*_s dB_s+(A^{*,+}_T-A^{*,+}_t)-(A^{*,-}_T-A^{*,-}_t),
	\end{split}\end{equation}
	where $A^{*,+}_t=A^{I,+}_t+\int_0^t g^-(s,Y^*_s,Z^*_s)d\langle B\rangle_s$ and $A^{*,-}_t=A^{I,-}_t+\int_0^t g^+(s,Y^*_s,Z^*_s)d\langle B\rangle_s$.  Consider the following two $G$-BSDEs:
	\begin{align}
		& Y_t^+=U_T+\int_t^T g(s,Y_s^+,Z_s^+)d\langle B\rangle_s+(A^{*,+}_T-A^{*,+}_t)-\int_t^T Z_s^+ dB_s-(K^+_T-K^+_t),\label{Y^+}\\
		& Y_t^-=L_T+\int_t^T g(s,Y_s^-,Z_s^-)d\langle B\rangle_s-(A^{*,-}_T-A^{*,-}_t)-\int_t^T Z_s^- dB_s-(K^-_T-K^-_t).\label{Y^-}
	\end{align}
    In fact, $(Y^++A^{*,+},Z^+,K^+)$ (resp., $(Y^--A^{*,-},Z^-,K^-)$) can be viewed as the solution to $G$-BSDE with terminal value $\tilde{\xi}^+$ (resp., $\tilde{\xi}^-$) and generator $\tilde{g}^+$ (resp., $\tilde{g}^-$), where 
    \begin{align*}
        &\tilde{\xi}^+=U_T+A^{*,+}_T, \ \tilde{g}^+(t,y,z)=g(t,y-A^{*,+}_t,z),\\
        &\tilde{\xi}^-=L_T-A^{*,-}_T,\ \tilde{g}^-(t,y,z)=g(t,y+A^{*,-}_t,z).
    \end{align*}
    It is easy to check that $(\tilde{\xi}^+,\tilde{g}^+)$ and $(\tilde{\xi}^-,\tilde{g}^-)$ satisfy (H1)-(H3). By Theorem \ref{wellposedness for GBSDE}, we have 
    \begin{equation}\label{y+y-}\begin{split}
        &\|Y^++A^{*,+}\|_{S^\infty_G}+\|Z^+\|_{BMO_G}\leq C(M_0,N_0,L_y,L_z,T), \\
        &\|Y^--A^{*,-}\|_{S^\infty_G}+\|Z^-\|_{BMO_G}\leq C(M_0,N_0,L_y,L_z,T).
    \end{split}\end{equation}
    By Theorem \ref{comparison theorem for GBSDE}, we have for any $t\in[0,T]$, $Y_t^-\leq Y_t^*\leq Y_t^+$, which implies that $Y_t^+\geq L_t$ and $Y_t^-\leq U_t$. Therefore, we may add the terms $+\int_t^T m(Y_s^+-L_s)^-ds$ and $-\int_t^T n(Y_s^--U_s)^+ds$ into Equations \eqref{Y^+} and \eqref{Y^-}, respectively. By Theorem \ref{comparison theorem for GBSDE} again, for any $t\in[0,T]$ and $n,m$, we have $Y_t^-\leq Y^{n,m}_t\leq Y^+_t$. Recalling \eqref{y+y-} and the definition for $A^{*,+}$, $A^{*-}$, we have 
	 $$\|Y^{n,m}\|_{S^\infty_G}\leq C(M_0,N_0,L_y,L_z,T).$$
\end{proof}

The following lemma provides the explicit convergence rate for $(Y^{n,m}-U)^+$, which is of vital importance in proving several fundamental estimates, such as the uniform estimates for the $G$-BMO norm of $Z^{n,m}$ and the convergence rate of $(\bar{Y}^n-U)^+$. The two estimates mentioned above are the challenging problems in doubly reflected $G$-BSDEs. Actually, for the lower obstacle case, i.e., $U\equiv +\infty$, the parameter $n$ has no effect in \eqref{Ynm}. We may drop it in the superscript. In this case, the uniform estimates of the $G$-BMO norm of $\{Z^m\}$ can be obtained naturally using the uniform estimate for $\{Y^m\}$ in Lemma \ref{est-Ynm} and Proposition \ref{est-Z-A} since $A^{m,+}-K^m$ is nondecreasing and $A^{m,-}\equiv 0$. However, this method is invalid for the doubly reflected case (actually, it is even invalid for the upper obstacle case). Second, if we only consider the penalized reflected $G$-BSDEs \eqref{penalized RGBSDE}, it is difficult to obtain the  convergence rate of $(\bar{Y}^n-U)^+$ since $\bar{A}^n$ is not a non-increasing $G$-martingale. 

\begin{lemma}\label{Ynm-U}
    There exists a constant $C$ independent of $n,m$, such that
	\begin{displaymath}
		n\sup_{t\in[0,T]}|(Y_t^{n,m}-U_t)^+|\leq C. 
	\end{displaymath}
\end{lemma}

\begin{proof}
	Consider the following $G$-BSDE:
	\begin{equation}\begin{split}\label{hatY^n}
			\hat{Y}^n_t=&U_T+\int_t^T g(s,\hat{Y}^n_s,\hat{Z}^n_s)d\langle B\rangle_s-\int_t^T n(\hat{Y}^n_s-U_s)^+ds\\ &+(A_T^{*,+}-A_t^{*,+})-\int_t^T \hat{Z}^n_sd B_s-(\hat{K}^n_T-\hat{K}^n_t),
	\end{split}\end{equation}
	where $A^{*,+}$ is the same as in the proof of Lemma \ref{est-Ynm}. For each $n$, $(\tilde{Y}^n,\hat{Z}^n,\hat{K}^n)$ can be viewed as the solution to $G$-BSDE with terminal value $\tilde{\xi}^+$ and generators $\tilde{f}^n$, $\tilde{g}^{+}$, where $\tilde{Y}^n_t=\hat{Y}^n_t+A^{*,+}_t$ and
    \begin{align*}
        \tilde{f}^{n}(t,y,z)=-n(y-A^{*,+}_t-U_t)^-.
    \end{align*}
     Here, $\tilde{\xi}^+$ and $\tilde{g}^+$ are the same as in the proof of Lemma \ref{est-Ynm}. It is easy to check that $(\tilde{\xi}^+,\tilde{f}^{n},\tilde{g}^+)$ satisfies (H1)-(H3). By Theorem \ref{wellposedness for GBSDE}, for any $n\geq1$, we have 
    \begin{align}\label{tildeYn hatKn}
        (\tilde{Y}^n,\hat{Z}^n,\hat{K}^n_T)\in S^\infty_G(0,T)\times BMO_G\times\bigcap_{p\geq 1}L^p_G(\Omega_T).
    \end{align}
    Consequently,
    \begin{align*}
        (\hat{Y}^n,\hat{Z}^n,\hat{K}^n_T,A^{*,+})\in S^\infty_G(0,T)\times BMO_G\times\bigcap_{p\geq 1}L^p_G(\Omega_T)\times \bigcap_{p\geq 1}S^p_G(0,T).
    \end{align*}

    Noting that $Y^*_t=I_t\leq U_t$, we may add the $-\int_t^T n(Y^*_s-U_s)^+ds$ term into Equation \eqref{Y^*}. By Theorem \ref{comparison theorem for GBSDE}, we have $ \hat{Y}^n_t\geq Y^*_t$ and hence $\hat{Y}^n_t\geq L_t$ for any $n\in\mathbb{N}$ and $t\in[0,T]$. Therefore, we may add the $+\int_t^T m(\hat{Y}^n_s-L_s)^-ds$ term into Equation \eqref{hatY^n}. Applying Theorem \ref{comparison theorem for GBSDE} again implies $\hat{Y}^n_t\geq Y^{n,m}_t$. Therefore, it suffices to prove that there exists a constant $C$ independent of $n,m$, such that
	\begin{displaymath}
		n(\hat{Y}^n_t-U_t)^+\leq C.
	\end{displaymath}

    For each given $\varepsilon>0$, let $l$ be a Lipschitz continuous function such that $I_{[-\varepsilon,\varepsilon]}(x)\leq l(x)\leq I_{[-2\varepsilon,2\varepsilon]}(x)$ for $x\in\mathbb{R}$. For any $s\in[0,T]$, set 
    \begin{align*}
        \tilde{a}^{n,\varepsilon}_s&:=[1-l(\tilde{Y}^n_s)]\frac{\tilde{g}^+(s,\tilde{Y}^n_s,\hat{Z}^n_s)-\tilde{g}^+(s,0,\hat{Z}^n_s)}{\tilde{Y}^{n}_s}I_{\{|\tilde{Y}^n_s|>0\}},\\
         \tilde{b}^{n,\varepsilon}_s&:=[1-l(\bar{Z}^n_s)]\frac{\tilde{g}^+(s,0,\hat{Z}^n_s)-\tilde{g}^+(s,0,\sigma(s))}{|\bar{Z}^{n}_s|^2}\bar{Z}^n_sI_{\{|\bar{Z}^n_s|>0\}},\\
        \tilde{m}^{n,\varepsilon}_s&:=l(\tilde{Y}^n_s)[\tilde{g}^+(s,\tilde{Y}^n_s,\hat{Z}^n_s)-\tilde{g}^+(s,0,\hat{Z}^n_s)]+l(\bar{Z}^n_s)[\tilde{g}^+(s,0,\hat{Z}^n_s)-\tilde{g}^+(s,0,\sigma(s))],
    \end{align*}
    where $\bar{Z}^n_s=\hat{Z}^n_s-\sigma(s)$. It is easy to check that 
    \begin{align*}
        |\tilde{a}^{n,\varepsilon}_s|\leq L_y, \ |\tilde{b}^{n,\varepsilon}_s|\leq L_z(1+|\hat{Z}^n_s|+|\sigma(s )|), \ |\tilde{m}^{n,\varepsilon}_s|\leq 2\varepsilon(L_y+L_z(1+2\varepsilon+2|\sigma(s)|)).
    \end{align*}
    Moreover, by Lemma 3.6 in \cite{HLSH}, $\tilde{b}^{n,\varepsilon}\in H^2_G(0,T)$ is a $G$-BMO martingale generator. Therefore, we may define a new conditional $G$-expectation $\widetilde{\mathbb{E}}^{n,\varepsilon}_t$ as follows
    \begin{align*}
        \widetilde{\mathbb{E}}^{n,\varepsilon}_t[X]:=\hat{\E}_t\left[\frac{\mathcal{E}(\tilde{b}^{n,\varepsilon})_T}{\mathcal{E}(\tilde{b}^{n,\varepsilon})_t}X\right].
    \end{align*}
    By Lemma \ref{lem2.6 2.7}, $\widetilde{B}^{n,\varepsilon}$ is a $G$-Brownian motion under $\widetilde{\E}^{n,\varepsilon}$, where $\widetilde{B}^{n,\varepsilon}_t=B_t-\int_0^t \tilde{b}^{n,\varepsilon}_sd\langle B\rangle_s$. Set $\tilde{U}_t=U_t+A^{*,+}_t$ and  $\bar{Y}^n_t=\tilde{Y}^n_t-\tilde{U}_t(=\hat{Y}^n_t-U_t)$. It is easy to check that  
	\begin{align*}
		\bar{Y}^n_t=&\tilde{\xi}^+-\tilde{U}_T+\int_t^Tb(s)ds+\int_t^T (\tilde{g}^+(s,0,\sigma(s))+\tilde{m}^{n,\varepsilon}_s+\tilde{a}^{n,\varepsilon}_s\tilde{Y}^n_s)d\langle B\rangle_s-\int_t^T \bar{Z}^n_sd\widetilde{B}^{n,\varepsilon}_s \notag\\
       & -\int_t^T g^-(s,I_s,\sigma^I_s)d\langle B\rangle_s-\int_t^T n(\bar{Y}^n_s)^+ds-(\hat{K}^n_T-\hat{K}_t^n)+(K^u_T-K^u_t).
	\end{align*}
    Applying It\^{o}'s formula to $e^{-nt}\bar{Y}^n_t$, we have
    \begin{equation}\label{entbaryn}\begin{split}
        &e^{-nt}\bar{Y}^n_t+\int_t^T e^{-ns}\bar{Z}^n_sd\widetilde{B}^{n,\varepsilon}_s+\int_t^T e^{-ns}d\hat{K}^n_s\\
        =&e^{-nT}(\tilde{\xi}^+-\tilde{U}_T)+\int_t^T e^{-ns}b(s)ds+n\int_t^T e^{-ns}(\bar{Y}^n_s-(\bar{Y}^n_s)^+)ds+\int_t^T e^{-ns}dK^u_s\\
        &+\int_t^Te^{-ns}\left(\tilde{g}^+(s,0,\sigma(s))+\tilde{m}^{n,\varepsilon}_s+\tilde{a}^{n,\varepsilon}_s\tilde{Y}^n_s-g^-(s,I_s,\sigma^I_s)\right)d\langle B\rangle_s\\
        \leq&\int_t^T e^{-ns}|b(s)|ds+\int_t^Te^{-ns}\left(2|g(s,0,0)|+L_z(1+|\sigma^I_s|)|\sigma^I_s|+L_z(1+|\sigma(s)|)|\sigma(s)|\right)d\langle B\rangle_s\\
        &+\int_t^Te^{-ns}\left(L_y\left(|A^{*,+}_s|+|I_s|+|\tilde{Y}^n_s|\right)+2\varepsilon\left(L_y+L_z(1+2\varepsilon+2|\sigma(s)|)\right)\right)d\langle B\rangle_s.
    \end{split}\end{equation}
     By Theorem \ref{comparison theorem for GBSDE}, for any $n\geq 1$ and any $t\in[0,T]$, we have $Y^+_t\geq \hat{Y}^n_t\geq Y^{n,m}_t$. By Equation \eqref{y+y-} and Lemma \ref{est-Ynm}, there exists a constant $C$ independent of $n$, such that 
    \begin{align*}
        \|\hat{Y}^n\|_{S^\infty_G}\leq C.
    \end{align*}
    Recalling that $\tilde{Y}^n_t=\hat{Y}^n_t+A^{*,+}_t$, it follows that 
    \begin{align*}
        \|\tilde{Y}^n\|_{S^\infty_G}\leq C.
    \end{align*}
 Recalling Equation \eqref{tildeYn hatKn}, by Lemma \ref{lem2.6 2.7}, for each $n\in\mathbb{N}$, $\hat{K}^n$ is a non-increasing $G$-martingale under $\widetilde{\E}^{n,\varepsilon}$. Then, by Lemma 3.4 in \cite{HJPS1}, $\{\int_0^t e^{-ns}d\hat{K}^n_s\}_{t\in[0,T]}$ is a non-increasing $G$-martingale under $\widetilde{\E}^{n,\varepsilon}$. Taking conditional expectations under $\widetilde{\E}^{n,\varepsilon}$ on both sides of \eqref{entbaryn}, we have
    \begin{align*}
        \bar{Y}^n_t\leq &\int_t^T e^{-n(s-t)}ds \Big[\|b\|_{S^\infty_G}+\bar{\sigma}^2\Big(2M_0+L_y\left(\|I\|_{S^\infty_G}+\|\tilde{Y}^n\|_{S^\infty_G}+\|A^{*,+}_T\|_{L^\infty_G}\right)\\
        &+L_z(1+\|\sigma\|_{S^\infty_G})\|\sigma\|_{S^\infty_G}+L_z(1+\|\sigma^I\|_{S^\infty_G})\|\sigma^I\|_{S^\infty_G}+2\varepsilon\left(L_y+L_z\left(1+2\varepsilon+2\|\sigma\|_{S^\infty_G}\right)\right)\Big)\Big].
        \end{align*}
    Then, we obtain the desired result.
\end{proof}

Now, we are ready to provide uniform estimates for $\{Z^{n,m}\}$, $\{A^{n,m,+}\}$ and $\{K^{n,m}\}$ under appropriate norm, respectively. 
\begin{lemma}\label{est-Znm-Anm}
   For any $\alpha\geq 1$, There exist two constants $C$, $C_\alpha$ independent of $m,n$, such that 
    \begin{align*}
        \|A^{n,m,-}\|_{L^\infty_G}\leq C,  \ \|Z^{n,m}\|_{BMO_G}\leq C, \  \hat{\E}[|A^{n,m,+}_T|^\alpha]\leq C_\alpha, \ \hat{\E}[|K^{n,m}_T|^\alpha]\leq C_\alpha.
    \end{align*}
\end{lemma}

\begin{proof}
    The first estimate is a direct consequence of Lemma \ref{Ynm-U}. By Proposition \ref{est-Z-A} and Lemma \ref{est-Ynm}, we obtain the remaining estimates. 
\end{proof}

By a similar analysis as the proof of Lemma 4.2, Lemma 4.3 and Theorem 5.1 in \cite{CT}, for any fixed $n$ and $\alpha\geq 2$, we have
\begin{description}
	\item[(a)] $\lim_{m\rightarrow\infty}\hat{\mathbb{E}}\left[\sup_{t\in[0,T]}|(Y^{n,m}_t-L_t)^-|^\alpha\right]=0$;
	\item[(b)]$(Y^{n,m},Z^{n,m},A^{n,m,+}-K^{n,m})$ converges to $(\bar{Y}^n,\bar{Z}^n,\bar{A}^n)\in S^\infty_G(0,T)\times BMO_G\times \cap_{\alpha\geq2}S^\alpha_G(0,T)$ in the following sense
	\begin{equation}\label{convergence}\begin{split}
		&\lim_{m\rightarrow\infty}\hat{\mathbb{E}}\left[\sup_{t\in[0,T]}|\bar{Y}^n_t-Y^{n,m}_t|^\alpha\right]=0, \\ &\lim_{m\rightarrow\infty}\hat{\mathbb{E}}\left[\left(\int_0^T|\bar{Z}^n_t-Z^{n,m}_t|^2dt\right)^{\alpha/2}\right]=0, \\ &\lim_{m\rightarrow\infty}\hat{\mathbb{E}}\left[\sup_{t\in[0,T]}|\bar{A}^n_t-(A^{n,m,+}_t-K^{n,m}_t)|^\alpha\right]=0.
	\end{split}\end{equation}
    Moreover, $(\bar{Y}^n,\bar{Z}^n,\bar{A^n})$ solves the following reflected $G$-BSDE
    \begin{equation}\label{barY^n}
	\begin{cases}
		\bar{Y}^n_t=\xi+\int_t^T g(s,\bar{Y}^n_s,\bar{Z}^n_s)d\langle B\rangle_s-n\int_t^T(\bar{Y}^n_s-U_s)^+ds-\int_t^T \bar{Z}^n_sdB_s+(\bar{A}^n_T-\bar{A}^n_t), \\
		\bar{Y}^n_t\geq L_t, \ \forall t\in[0,T],\\
        \{-\int_0^t (\bar{Y}^n_s-L_s)d\bar{A}^n_s\}_{t\in[0,T]} \textrm{ is a non-increasing $G$-martingale}.
	\end{cases}
\end{equation}
    \end{description}

    The following result is an extension of Lemma 3.6 in \cite{LN}, which provides the uniform estimates for $\{\bar{Y}^n\}$, $\{\bar{Z}^n\}$ and $\{\bar{A}^n\}$ under suitable norm and the explicit convergence rate for $(\bar{Y}^n-U)^+$. 
    \begin{lemma}\label{est-barYn-barZn}
       For any $\alpha\ge 2$, there exist two constant $C,C_\alpha$ independent of $n$, such that 
       \begin{align*}
           &\|\bar{Y}^n\|_{S^\infty_G}\leq C, \ n\sup_{t\in[0,T]}(\bar{Y}^n_t-U_t)^+ \leq C,\\
           &\hat{\E}[|\bar{A}^n_T|^\alpha]\leq C_\alpha, \  \|\bar{Z}^n\|_{BMO_G}\leq C.
       \end{align*}
    \end{lemma}

    \begin{proof}
      The first two estimates can be proved similarly as the proof of Theorem 5.1 in \cite{CT}. For readers' convenience, we give a short proof here for the first estimate. Recalling Theorem \ref{the1.1} and Equation \eqref{convergence}, for each $\P\in\mathcal{P}$, the sequence $\{\sup_{t\in[0,T]}|Y^{n,m}_t|\}_{m=1}^\infty$ converges in probability $\P$ to $\sup_{t\in[0,T]}|\bar{Y}^n_t|$. Then, there exists a subsequence $\{\sup_{t\in[0,T]}|Y^{n,m_k}_t|\}_{k=1}^\infty$ such that 
      \begin{align*}
          \lim_{k\rightarrow \infty}\sup_{t\in[0,T]}|Y^{n,m_k}_t|=\sup_{t\in[0,T]}|\bar{Y}^n_t|, \  \P\textrm{-a.s.}
      \end{align*}
      Recalling Lemma \ref{est-Ynm}, there exists a constant $C$ independent of $n,\P$, such that $\sup_{t\in[0,T]}|\bar{Y}^n_t|\leq C$, $\P$-a.s. for each $\P\in\mathcal{P}$, which indicates that $\sup_{t\in[0,T]}|\bar{Y}^n_t|\leq C$, q.s. and the first estimate follows. 
      
      The third estimate is a direct consequence of Equation \eqref{convergence} and Lemma \ref{est-Znm-Anm}. The last one can be obtained by applying Proposition \ref{est-Z-A} and the first two estimates. 
    \end{proof}

    \subsection{Proof of Theorem \ref{main1}}

    In this subsection, we will prove Theorem \ref{main1} as follows.

    \textbf{Step 1.} Show the uniqueness of the solution.

    \textbf{Step 2.} Show the doubly reflected $G$-BSDEs has a solution $(Y,Z,A)$, which is the limit of $(Y^n,Z^n,A^n)$, where 
    \begin{align*}
        Y^n=Y^{n,n},\  Z^n=Z^{n,n}, \ K^n=K^{n,n}, \ A^{n,+}=A^{n,n,+},  \ A^{n,-}=A^{n,n,-}
    \end{align*}
    and 
    \begin{align*}
        A^n=A^{n,-}-K^n-A^{n,+}.
    \end{align*}

    \textbf{Step 3.} Show that for any $\alpha\geq 2$, 
    \begin{align}\label{Yn barYn}
        \lim_{n\rightarrow\infty}\hat{\mathbb{E}}\left[\sup_{t\in[0,T]}|Y^n_t-\bar{Y}^n_t|^\alpha\right]=0.
    \end{align}

\begin{proof}[Proof of Step 1]
    Let $(Y^i,Z^i,A^i)$, $i=1,2$ be the solutions to the doubly reflected $G$-BSDE with parameters $(\xi,g,L,U)$. By Proposition \ref{est-Y1-Y2}, we conclude that $Y^1\equiv Y^2$. Applying It\^{o}'s formula to $(Y^1_t-Y^2_t)^2$, for any $\tau\in\mathcal{T}_0^T$ and $\P\in\mathcal{P}$, the following equation holds $\P$-a.s.,  
 \begin{align*}
     \int_\tau^T |Z^1_s-Z^2_s|^2d\langle B\rangle_s=&-(Y^1_\tau-Y^2_\tau)^2+\int_\tau^T (Y^1_s-Y^2_s)(g(s,Y^1_s,Z^1_s)-g(s,Y^2_s,Z^2_s))d\langle B\rangle_s\\
     &-\int_\tau^T 2(Y^1_s-Y^2_s)(Z^1_s-Z^2_s)dB_s+\int_\tau^T 2(Y^1_s-Y^2_s)d(A^1_s-A^2_s).
 \end{align*}
 Using the fact that $Y^1\equiv Y^2$,  it is easy to check that 
 \begin{align*}
     \|Z^1_s-Z^2_s\|_{BMO_G}=0.
 \end{align*}
 Since for $i=1,2$,
 \begin{align*}
     A^i_t=Y^i_0-Y^i_t-\int_0^t g(s,Y^i_s,Z^i_s)d\langle B\rangle_s+\int_0^t Z^i_s dB_s,
 \end{align*}
 For any $\alpha\geq 2$, applying the continuity property of $g$, the H\"{o}lder inequality and Lemma \ref{lem2.2}, there exists a constant $C$ depending on $\alpha,L_y,L_z,T$, such that 
 \begin{align*}
     \hE\left[\sup_{t\in[0,T]}|A^1_t-A^2_t|^\alpha\right]\leq& C\Bigg(\hE\left[\sup_{t\in[0,T]}|Y^1_t-Y^2_t|^\alpha\right]+\hE\left[\left(\int_0^T |Z^1_s-Z^2_s|^2d\langle B\rangle_s\right)^{\frac{\alpha}{2}}\right]\\
     &+\hE\left[\left(\int_0^T \left(|Y^1_s-Y^2_s|+(1+|Z^1_s|+|Z^2_s|)|Z^1_s-Z^2_s|\right)d\langle B\rangle_s\right)^{\alpha}\right]\Bigg)\\
     \leq &C\bigg(\|Y^1-Y^2\|_{S^\infty_G}^\alpha+\|Z^1-Z^2\|_{BMO_G}^\alpha\\
     &+\left(1+\|Z^1\|_{BMO_G}^\alpha+\|Z^2\|_{BMO_G}^\alpha\right)\|Z^1-Z^2\|_{BMO_G}^\alpha\bigg).
 \end{align*}
 Therefore, we have $A^1\equiv A^2$.
\end{proof}

Now, we are in a position to proceed Step 2, which will be divided into several lemmas. We first show that $(Y^n-L)^-$ converges to $0$ uniformly in $t$, which is a generalization of Lemma 4.2 in \cite{CT} and Lemma 4.4 in \cite{LS}. It should be pointed out that the proof of Lemma 4.2 in \cite{CT} cannot be applied here since $Y^n$ is not monotone in $n$ in the present framework. Fortunately, by applying some continuity property for the obstacle $L$ (see Lemma 4.2 in \cite{LS}), our proof is even more concise compared with the one of Lemma 4.2 in \cite{CT}.
\begin{lemma}\label{Yn-L}
    For any $\alpha\geq 2$, we have
        \begin{align}\label{EYn-L}
            \lim_{n\rightarrow \infty}\hE\left[\sup_{t\in[0,T]}|(Y^n_t-L_t)^-|^\alpha\right]=0.
        \end{align}
\end{lemma}

 \begin{proof}
        By Lemma \ref{est-Ynm}, Lemma \ref{Ynm-U} and Lemma \ref{est-Znm-Anm}, for any $\alpha\geq 2$, there exist two constants $C,C_\alpha$ independent of $n$, such that 
        \begin{equation}\label{est-YnZnKn}\begin{split}
          &\|Y^n\|_{S^\infty_G}\leq C,\  n\sup_{t\in[0,T]}(Y^n_t-U_t)^+\leq C,\   \|A^{n,-}\|_{L^\infty_G}\leq C, \\
          & \|Z^{n}\|_{BMO_G}\leq C, \  \hat{\E}[|A^{n,+}_T|^\alpha]\leq C_\alpha, \ \hat{\E}[|K^{n}_T|^\alpha]\leq C_\alpha.
        \end{split}\end{equation}

        Similar to the proof of Proposition \ref{est-Y1-Y2} and Lemma \ref{Ynm-U}, we first rewrite the generator as follows
        \begin{align*}
            g(s,Y^n_s,Z^n_s)=g(s,0,0)+m^{n,\varepsilon}_s+a^{n,\varepsilon}_sY^n_s+b^{n,\varepsilon}_sZ^n_s,
        \end{align*}
        where 
        \begin{align*}
             |{a}^{n,\varepsilon}_s|\leq L_y, \ |{b}^{n,\varepsilon}_s|\leq L_z(1+|{Z}^n_s|), \ |{m}^{n,\varepsilon}_s|\leq 2\varepsilon(L_y+L_z(1+2\varepsilon)).
        \end{align*}
        Therefore, $G$-BSDE \eqref{Ynm} can be rewritten as 
        \begin{align*}
            Y^{n}_t=&\xi+\int_t^T \left(g(s,0,0)+m^{n,\varepsilon}_s+a^{n,\varepsilon}_sY^n_s\right)d\langle B\rangle_s-\int_t^T Z_s^{n}d{B}^{n,\varepsilon}_s\\
        &-(K_T^{n}-K_t^{n})
		+\int_t^T n(Y_s^{n}-L_s)^-ds-\int_t^T n(Y_s^{n}-U_s)^+ds,
        \end{align*}
        where $B^{n,\varepsilon}_t=B_t-\int_0^t b^{n,\varepsilon}_sd\langle B\rangle_s$. By Lemma 3.6 in \cite{HLSH} and \eqref{est-YnZnKn}, $b^{n,\varepsilon}\in H^2_G(0,T)$ is a $G$-BMO martingale generator. Therefore, we can define a new $G$-expectation $\hE^{n,\varepsilon}$ by $\mathcal{E}(b^{n,\varepsilon})$ and $B^{n,\varepsilon}$ is a $G$-Brownian motion under $\hat{\E}^{n,\varepsilon}$.

        In the following, we first prove that for any $\alpha\geq 2$
        \begin{align}\label{Envarepsilon Yn-L}
            \lim_{n\rightarrow\infty}\hE^{n,\varepsilon}\left[\sup_{t\in[0,T]}|(Y^n_t-L_t)^-|^\alpha\right]=0.
        \end{align}
        For this purpose, let $(y^n,z^n,k^n)$ be the solution to the following $G$-BSDE
        \begin{align*}
            y^{n}_t=&\xi+\int_t^T \left(g(s,0,0)+m^{n,\varepsilon}_s+a^{n,\varepsilon}_sY^n_s\right)d\langle B\rangle_s-\int_t^T z_s^{n}d{B}^{n,\varepsilon}_s\\
        &-(k_T^{n}-k_t^{n})
		+\int_t^T n(L_s-y^n_s)ds-\int_t^T n(Y_s^{n}-U_s)^+ds.
        \end{align*}
        We can solve it explicitly to obtain that
        \begin{align*}
            y^n_t=e^{nt}\hE^{n,\varepsilon}_t&\Bigg[e^{-nT}\xi+\int_t^T n e^{-ns}L_sds-\int_t^T n e^{-ns}(Y^n_s-U_s)^+ ds\\
            &+\int_t^T e^{-ns}\left(g(s,0,0)+m^{n,\varepsilon}_s+a^{n,\varepsilon}_sY^n_s\right)d\langle B\rangle_s\Bigg].
        \end{align*}
        In view of Theorem 3.6 in \cite{HJPS2} (i.e., the comparison theorem for $G$-BSDEs with Lipschitz generators) and noting that $\xi\geq L_T$, for any $t\in[0,T]$, we have 
        \begin{align*}
            Y^n_t-L_t\geq y^n_t-L_t=\hE^{n,\varepsilon}_t&\Bigg[\widetilde{L}^n_t-\int_t^T n e^{n(t-s)}(Y^n_s-U_s)^+ ds\\
            &+\int_t^T e^{n(t-s)}\left(g(s,0,0)+m^{n,\varepsilon}_s+a^{n,\varepsilon}_sY^n_s\right)d\langle B\rangle_s\Bigg],
        \end{align*}
        where 
        \begin{align*}
            \widetilde{L}^n_t=e^{n(t-T)}(L_T-L_t)+\int_t^T n e^{n(t-s)}(L_s-L_t)ds.
        \end{align*}
        Consequently, we have
        \begin{align*}
            (Y^n_t-L_t)^-\leq \hE^{n,\varepsilon}_t&\Bigg[|\widetilde{L}^n_t|+\int_t^T n e^{n(t-s)}(Y^n_s-U_s)^+ ds\\
            &+\left|\int_t^T e^{n(t-s)}\left(g(s,0,0)+m^{n,\varepsilon}_s+a^{n,\varepsilon}_sY^n_s\right)d\langle B\rangle_s\right|\Bigg].
        \end{align*}
        Similar analysis as the proof for Equation (4.4) in \cite{CT}, for any $\alpha\geq 2$, we have
         \begin{align}\label{ee1}
            \lim_{n\rightarrow\infty}\hE^{n,\varepsilon}\left[\sup_{t\in[0,T]}\left|\int_t^T e^{n(t-s)}\left(g(s,0,0)+m^{n,\varepsilon}_s+a^{n,\varepsilon}_sY^n_s\right)d\langle B\rangle_s\right|^\alpha\right]=0.
        \end{align}
       Using \eqref{est-YnZnKn}, simple calculation yields that for any $\alpha\geq 2$,
        \begin{align}\label{ee2}
            \lim_{n\rightarrow\infty}\hE^{n,\varepsilon}\left[\sup_{t\in[0,T]}\left|\int_t^T n e^{n(t-s)}(Y^n_s-U_s)^+ ds\right|^\alpha\right]\leq \lim_{n\rightarrow\infty}\frac{C}{n^{\alpha}}=0. 
            \end{align}

        By \eqref{est-YnZnKn}, there exists a constant $p>1$ independent of $n,\varepsilon$, such that 
        \begin{align*}
            \|b^{n,\varepsilon}\|_{BMO_G}\leq L_z(1+\|Z^n\|_{BMO_G})<\phi(p).
        \end{align*}
        Set $q=\frac{p}{p-1}$. By Lemma \ref{lem2.4 2.5}, for any $X\in L^q_G(\Omega_T)$, we have
        \begin{align}\label{holder for Envarepsilon}
            \hE^{n,\varepsilon}_t[X]\leq \left(\hE\left[\frac{\mathcal{E}(b^{n,\varepsilon})_T^p}{\mathcal{E}(b^{n,\varepsilon})_t^p}\right]\right)^{\frac{1}{p}}\left(\hE_t[|X|^q]\right)^{\frac{1}{q}}\leq C_p\left(\hE_t[|X|^q]\right)^{\frac{1}{q}}.
        \end{align}
        Consequently, there exists a constant $C'$ independent of $n,\varepsilon$, such that for any $\alpha\ge 2$,
        \begin{align}\label{doob for Envarepsilon'}
            \hE^{n,\varepsilon}\left[\sup_{t\in[0,T]}\hE^{n,\varepsilon}_t\left[|\widetilde{L}^n_t|^\alpha\right]\right]\leq C'\left(\hE\left[\sup_{t\in[0,T]}\hE_t\left[|\widetilde{L}^n_t|^{\alpha q}\right]\right]\right)^{\frac{1}{q}}.
        \end{align}
        Moreover, in view of Remark 2.9 in \cite{HJPS1}, there exists a constant $C''$ independent of $n,\varepsilon$, such that for any $\alpha\geq 2$,
        \begin{align}\label{doob for Envarepsilon}
            \hE^{n,\varepsilon}\left[\sup_{t\in[0,T]}\hE^{n,\varepsilon}_t\left[|\widetilde{L}^n_t|^\alpha\right]\right]\leq C''\left(\hE^{n,\varepsilon}\left[\sup_{t\in[0,T]}|\widetilde{L}^n_t|^{2\alpha}\right]+\left(\hE^{n,\varepsilon}\left[\sup_{t\in[0,T]}|\widetilde{L}^n_t|^{2\alpha}\right]\right)^{\frac{1}{2}}\right).
        \end{align}
        Since $L\in \cap_{\alpha\geq 2} S^\alpha_G(0,T)$, by Lemma 4.2 in \cite{LS}, for any $\alpha\geq 2$, we have 
        \begin{equation}\label{widetildeL}
            \lim_{n\rightarrow\infty}\hE\left[\sup_{t\in[0,T]}\hE_t\left[|\widetilde{L}^n_t|^\alpha\right]\right]=0.
        \end{equation}
        Combining Equations \eqref{doob for Envarepsilon'} and \eqref{widetildeL}, for any $\alpha\geq 2$, we have 
        \begin{align*}
            \lim_{n\rightarrow\infty}\hE^{n,\varepsilon}\left[\sup_{t\in[0,T]}\hE^{n,\varepsilon}_t\left[|\widetilde{L}^n_t|^\alpha\right]\right]=0.
        \end{align*}
        Besides, by Equations \eqref{ee1}, \eqref{ee2} and \eqref{doob for Envarepsilon}, for any $\alpha\geq 2$, we have
        \begin{align*}
            &\lim_{n\rightarrow\infty}\hE^{n,\varepsilon}\left[\sup_{t\in[0,T]}\hE^{n,\varepsilon}_t\left[\left|\int_t^T e^{n(t-s)}\left(g(s,0,0)+m^{n,\varepsilon}_s+a^{n,\varepsilon}_sY^n_s\right)d\langle B\rangle_s\right|^\alpha\right]\right]=0,\\
            &\lim_{n\rightarrow\infty}\hE^{n,\varepsilon}\left[\sup_{t\in[0,T]}\hE^{n,\varepsilon}_t\left[\left|\int_t^T n e^{n(t-s)}(Y^n_s-U_s)^+ ds\right|^\alpha\right]\right]=0.
        \end{align*}
        All the above analysis indicates that Equation \eqref{Envarepsilon Yn-L} holds. 
        
        By Lemma \ref{lem2.4 2.5} and Remark 2.7 in \cite{CT}, there exists $r>1$ which is independent of $n,\varepsilon$, such that 
        \begin{align*}
            \hE\left[\left(\mathcal{E}(b^{n,\varepsilon}_T)\right)^{\frac{1}{1-r}}\right]\leq C_r,
        \end{align*}
        where $C_r$ is a constant depends on $r$ but not depend on $n,\varepsilon$. Therefore, for any $\alpha\geq 2$, it is easy to check that 
        \begin{align*}
            \hE\left[\sup_{t\in[0,T]}|(Y^n_t-L_t)^-|^\alpha\right]&\leq \left(\hE\left[\mathcal{E}(b^{n,\varepsilon}_T)\sup_{t\in[0,T]}|(Y^n_t-L_t)^-|^{\alpha r}\right]\right)^{\frac{1}{r}} \left(\hE\left[\left(\mathcal{E}(b^{n,\varepsilon}_T)\right)^{\frac{1}{1-r}}\right]\right)^{\frac{r-1}{r}}\\
            &\leq C^{\frac{r-1}{r}}\left(\hE^{n,\varepsilon} \left[\sup_{t\in[0,T]}|(Y^n_t-L_t)^-|^{\alpha r}\right]\right)^{\frac{1}{r}}.
        \end{align*}
        Finally, we obtain that Equation \eqref{EYn-L} holds.
        \end{proof}

By a similar ananlysis as the proof of Lemma 4.3 and Theorem 5.1 in \cite{CT}, we have the following result.
\begin{lemma}\label{Cauchy}
    For any $\alpha\geq 2$, we have 
    \begin{align*}
        &\lim_{n,m\rightarrow \infty}\hE\left[\sup_{t\in[0,T]}|Y^n_t-Y^m_t|^\alpha\right]=0,\\
        &\lim_{n,m\rightarrow \infty}\hE\left[\left(\int_0^T|Z^n_t-Z^m_t|^2d\langle B\rangle_t\right)^{\frac{\alpha}{2}}\right]=0,\\
        &\lim_{n,m\rightarrow \infty}\hE\left[\sup_{t\in[0,T]}|A^n_t-A^m_t|^\alpha\right]=0.
    \end{align*}
\end{lemma}

Now, we are ready to show that the doubly reflected $G$-BSDE with quadratic generator has a solution.

\begin{proof}[Proof of Step 2]
    By Lemma \ref{Cauchy}, for any $\alpha\geq 2$, there exists $(Y,Z,A)\in {S}^\alpha_G(0,T)\times {H}^\alpha_G(0,T)\times {S}^\alpha_G(0,T)$, such that
    \begin{align*}
        &\lim_{n\rightarrow \infty}\hE\left[\sup_{t\in[0,T]}|Y^n_t-Y_t|^\alpha\right]=0,\\
        &\lim_{n\rightarrow \infty}\hE\left[\left(\int_0^T|Z^n_t-Z_t|^2d\langle B\rangle_t\right)^{\frac{\alpha}{2}}\right]=0,\\
        &\lim_{n\rightarrow \infty}\hE\left[\sup_{t\in[0,T]}|A^n_t-A_t|^\alpha\right]=0.
    \end{align*}
    By Equation \eqref{est-YnZnKn} and a similar analysis as the proof for Lemma \ref{est-barYn-barZn}, there exists a constant $C$ independent of $n$, such that 
        \begin{equation*}\begin{split}
          &\|Y\|_{S^\infty_G}\leq C.
        \end{split}\end{equation*}
       For any $\tau\in \mathcal{T}_0^T$ and any $\P\in\mathcal{P}$, we may check that 
        \begin{align*}
            \lim_{n\rightarrow\infty}\E^\P\left[\left|\E^\P_\tau\left[\int_\tau^T|Z^n_s|^2d\langle B\rangle_s\right]-\E^\P_\tau\left[\int_\tau^T|Z_s|^2d\langle B\rangle_s\right]\right|\right]=0,
        \end{align*}
        which amounts to say that 
        \begin{align*}
            \E^\P_\tau\left[\int_\tau^T|Z^n_s|^2d\langle B\rangle_s\right]\xrightarrow{\P} \E^\P_\tau\left[\int_\tau^T|Z_s|^2d\langle B\rangle_s\right], \textrm{ as } n\rightarrow\infty.
        \end{align*}
        Then, there exists a subsequence such that 
        \begin{align*}
            \lim_{k\rightarrow\infty}\E^\P_\tau\left[\int_\tau^T|Z^{n_k}_s|^2d\langle B\rangle_s\right]=\E^\P_\tau\left[\int_\tau^T|Z_s|^2d\langle B\rangle_s\right], \ \P\textrm{-a.s.}
        \end{align*}
        Therefore, by Equation \eqref{est-YnZnKn}, for any $\tau\in \mathcal{T}_0^T$ and any $\P\in\mathcal{P}$, we have 
        \begin{align*}
            \E^\P_\tau\left[\int_\tau^T|Z_s|^2d\langle B\rangle_s\right]\leq \sup_{n}\|Z^n\|_{BMO_G}\leq C.
        \end{align*}
        That is, $Z\in BMO_G$.  

        By Lemma \ref{Yn-L} and Equation \eqref{est-YnZnKn}, for any $t\in[0,T]$, we have $L_t\leq Y_t\leq U_t$. It remains to prove that $(Y,A)$ satisfies the approximate Skorohod condition. That is, for any $\alpha\geq 2$,       $$\lim\limits_{n\rightarrow\infty}\hat{\mathbb{E}}\left[\left|\int_0^T (Y_s-L_s)d A_s^{n,+}\right|^{\alpha/2}\right]=0,\ \lim\limits_{n\rightarrow\infty}\hat{\mathbb{E}}\left[\left|\int_0^T (U_s-Y_s)d A_s^{n,-}\right|^{\alpha/2}\right]=0.$$
        We only prove the first one. It is easy to check that 
        \begin{align*}
            \int_0^T (Y_s-L_s)d A_s^{n,+}&=\int_0^T (Y_s-Y^n_s)d A_s^{n,+}+\int_0^T (Y^n_s-L_s)n(Y^n_s-L_s)^-ds\\
            &\leq \sup_{t\in[0,T]}|Y_t-Y^n_t||A^{n,+}_T|.
        \end{align*}
        It follows that 
        \begin{align*}
           \lim\limits_{n\rightarrow\infty}\hat{\mathbb{E}}\left[\left|\int_0^T (Y_s-L_s)d A_s^{n,+}\right|^{\alpha/2}\right]\leq \lim_{n\rightarrow\infty} \left(\hE\left[\sup_{t\in[0,T]}|Y^n_t-Y_t|^\alpha\right]\right)^{\frac{1}{2}}\left(\hE[|A^{n,+}_T|^\alpha]\right)^{\frac{1}{2}}=0.
        \end{align*}
        The proof is complete.
\end{proof}

Finally, we show that the solution of doubly reflected $G$-BSDE can also be approximated by the solutions to a family of penalized reflected $G$-BSDEs, i.e., Equation \eqref{Yn barYn} holds.

\begin{proof}[Proof of Step 3]
    First, note that for any $n\in\mathbb{N}$ and any $t\in[0,T]$, $\bar{Y}^n_t\geq L_t$. Then, we have 
    \begin{align*}
		\bar{Y}_t^n=&\xi+\int_t^T g(s,\bar{Y}^n_s,\bar{Z}_s^n)d\langle B\rangle_s-\int_t^T \bar{Z}_s^ndB_s+(\bar{A}^n_T-\bar{A}^n_t)\\
		&-\int_t^T n(\bar{Y}^n_s-U_s)^+ds+\int_t^Tn(\bar{Y}^n_s-L_s)^-ds.
	\end{align*}
    By Theorem \ref{comparison theorem for GBSDE}, for any $n\in\mathbb{N}$ and any $t\in[0,T]$, we have $\bar{Y}^n_t\geq Y^n_t$. We define
    \begin{align*}
        &\widehat{Y}^n_t:=\bar{Y}^n_t-Y^n_t, \ \widehat{Z}^n_t:=\bar{Z}^n_t-Z^n_t, \ \widehat{K}^n_t:=-\bar{A}^n_t-K^n_t, \\
        &\widehat{A}^{n,+}_t:=\bar{A}^{n,+}_t-A^{n,+}_t, \ \widehat{A}^{n,-}_t:=\bar{A}^{n,-}_t-A^{n,-}_t,
    \end{align*}
    where
    \begin{align*}
        \bar{A}^{n,+}_t=\int_0^t n(\bar{Y}^n_s-L_s)^-ds(=0), \ 
        \bar{A}^{n,-}_t=\int_0^t n(\bar{Y}^n_s-U_s)^+ds.
    \end{align*}
    Similar as the proof of Proposition \ref{est-Y1-Y2}, for any $\varepsilon>0$ and any $s\in[0,T]$, we have
\begin{align*}
        g(s,\bar{Y}^n_s,\bar{Z}^n_s)-g(s,Y^n_s,Z^n_s)=\widehat{m}^{n,\varepsilon}_s+\widehat{a}^{n,\varepsilon}_s \widehat{Y}^n_s+\widehat{b}^{n,\varepsilon}_s \widehat{Z}^n_s
    \end{align*}
    and
    \begin{align*}
        |\widehat{a}^{n,\varepsilon}_s|\leq L_y, \ |\widehat{b}^{n,\varepsilon}_s|\leq L_z(1+|{Z}^n_s|+|\bar{Z}^n_s|), \ |\widehat{m}^{n,\varepsilon}_s|\leq 2\varepsilon(L_y+L_z(1+2\varepsilon+2|Z^n_s|)).
    \end{align*}
    Therefore, we have 
    \begin{align*}
        \widehat{Y}^n_t=\int_t^T(\widehat{m}^{n,\varepsilon}_s+\widehat{a}^{n,\varepsilon}_s \widehat{Y}^n_s)d\langle B\rangle_s-\int_t^T \widehat{Z}^n_sd\widehat{B}^{n,\varepsilon}_s-(\widehat{K}^n_T-\widehat{K}^n_t)+(\widehat{A}^{n,+}_T-\widehat{A}^{n,+}_t)-(\widehat{A}^{n,-}_T-\widehat{A}^{n,-}_t),
    \end{align*}
    where $\widehat{B}^{n,\varepsilon}_t=B_t-\int_0^t \widehat{b}^{n,\varepsilon}_sd\langle B\rangle_s$.

    For any constant $r>0$, applying $G$-It\^{o}'s formula to $e^{rt}(H^n_t)^{\frac{\alpha}{2}}$, where $H^n_t=|\widehat{Y}^n_t|^2$, we have
	\begin{equation}\label{e1}
		\begin{split}
			& |\widehat{Y}^n_t|^\alpha e^{rt}+\int_t^T re^{rs}|\widehat{Y}^n_s|^\alpha ds+\int_t^T \frac{\alpha}{2}(\alpha-1) e^{rs}
			|\widehat{Y}^n_s|^{\alpha-2}|\widehat{Z}^n_s|^2d\langle B\rangle_s\\
			=&
			\int_t^T{\alpha} e^{rs}|\widehat{Y}^n_s|^{\alpha-2}\widehat{Y}^n_s(\widehat{m}^{n,\varepsilon}_s+\widehat{a}^{n,\varepsilon}_s \widehat{Y}^n_s)d\langle B\rangle_s-\int_t^T\alpha e^{rs}|\widehat{Y}^n_s|^{\alpha-2}\widehat{Y}^n_s  n(Y_s^n-L_s)^-ds\\
            &-\int_t^T\alpha e^{rs}|\widehat{Y}^n_s|^{\alpha-2}(\widehat{Y}^n_s\widehat{Z}^n_sd\widehat{B}^{n,\varepsilon}_s-\widehat{Y}^n_sd{K}^n_s-\widehat{Y}^n_sd\bar{A}^n_s)\\
			&-\int_t^T \alpha e^{rs}|\widehat{Y}^n_s|^{\alpha-2}\widehat{Y}^n_s  n[(\bar{Y}^n_s-U_s)^+-(Y^n_s-U_s)^+]ds,
		\end{split}
	\end{equation}
    Noting that $\widehat{Y}^n_t\geq0$, it is easy to check that
	\begin{equation}\label{e2}\begin{split}
			&+\int_t^T\alpha e^{rs}|\widehat{Y}^n_s|^{\alpha-2}\widehat{Y}^n_sd\bar{A}^n_s\leq \int_t^T\alpha e^{rs}|\widehat{Y}^n_s|^{\alpha-2}[(\bar{Y}^n_s-L_s)+(Y^n_s-L_s)^-]d\bar{A}^n_s,\\
			&-\int_t^T \alpha e^{rs}|\widehat{Y}^n_s|^{\alpha-2}\widehat{Y}^n_s  n[(\bar{Y}^n_s-U_s)^+-(Y^n_s-U_s)^+]ds\leq 0,\\
			&-\int_t^T\alpha e^{rs}|\widehat{Y}^n_s|^{\alpha-2}\widehat{Y}^n_s  n(Y_s^n-L_s)^-ds\leq 0,\\
			&+\int_t^T\alpha e^{rs}|\widehat{Y}^n_s|^{\alpha-2}\widehat{Y}^n_sd{K}^n_s\leq 0.
	\end{split} \end{equation}
	Set $M^{n,\varepsilon}_t=\int_0^t\alpha e^{rs}|\widehat{Y}^n_s|^{\alpha-2}(\widehat{Y}^n_s\widehat{Z}^n_sd\widehat{B}^{n,\varepsilon}_s-(\bar{Y}^n_s-L_s)d\bar{A}^n_s)$. Let $r>L_y \alpha \bar{\sigma}^2$. Combining Equations \eqref{e1} and \eqref{e2} implies that 
    \begin{align*}
        |\widehat{Y}^n_t|^\alpha e^{rt}+M^{n,\varepsilon}_T-M^{n,\varepsilon}_t\leq \int_t^T{\alpha} e^{rs}|\widehat{Y}^n_s|^{\alpha-2}\widehat{Y}^n_s\widehat{m}^{n,\varepsilon}_sd\langle B\rangle_s+\int_t^T\alpha e^{rs}|\widehat{Y}^n_s|^{\alpha-2}(Y^n_s-L_s)^-d\bar{A}^n_s.
    \end{align*}
    Note that $M^{n,\varepsilon}$ is a $G$-martingale under $\widehat{\E}^{n,\varepsilon}$. Therefore, we obtain
    \begin{align}\label{widetildeYn}
        |\widehat{Y}^n_t|^\alpha e^{rt}\leq \widehat{\E}^{n,\varepsilon}_t\left[\int_t^T{\alpha} e^{rs}|\widehat{Y}^n_s|^{\alpha-2}\widehat{Y}^n_s\widehat{m}^{n,\varepsilon}_sd\langle B\rangle_s\right]+\widehat{\E}^{n,\varepsilon}_t\left[\int_t^T\alpha e^{rs}|\widehat{Y}^n_s|^{\alpha-2}(Y^n_s-L_s)^-d\bar{A}^n_s\right].
    \end{align}

    In view of Lemma \ref{est-barYn-barZn} and Equation \eqref{est-YnZnKn}, we can choose $p>1$ independent of $n,\varepsilon$, such that 
    \begin{align*}
        \|\widehat{b}^{n,\varepsilon}\|_{BMO_G}\leq L_z(1+\|Z^n\|_{BMO_G}+\|\bar{Z}^n\|_{BMO_G})<\phi(p).
    \end{align*}
    Set $q=\frac{p}{p-1}$. There exists $C_p>0$ depends on $p$ but not on $n,\varepsilon$, such that for any $X\in L^q_G(\Omega_T)$,
    \begin{align*}
        \widehat{\E}_t^{n,\varepsilon}[X]\leq C_q \left(\hE_t[|X|^q]\right)^{\frac{1}{q}}.
    \end{align*}
    Moreover, by Remark 2.9 in \cite{HJPS1}, for some $\beta>1$, there exists a constant $C'$ independent of $n,\varepsilon$, such that 
        \begin{align}\label{doob for widehatEnvarepsilon}
            \widehat{\E}^{n,\varepsilon}\left[\sup_{t\in[0,T]}\widehat{\E}^{n,\varepsilon}_t\left[|X|\right]\right]\leq C'\left(\widehat{\E}^{n,\varepsilon}\left[\sup_{t\in[0,T]}|X|^{\beta}\right]+\left(\widehat{\E}^{n,\varepsilon}\left[\sup_{t\in[0,T]}|X|^{\beta}\right]\right)^{\frac{1}{\beta}}\right).
        \end{align}
    By Lemma \ref{lem2.4 2.5}, Remark 2.7 in \cite{CT} and the fact that $\|\widehat{b}^{n,\varepsilon}\|_{BMO_G}\leq C$, there is $r>1$ which is independent of $n,\varepsilon$, such that 
    \begin{align*}
        \hE\left[\left(\mathcal{E}(\widehat{b}^{n,\varepsilon})_T\right)^{\frac{1}{1-r}}\right]\le C_r.
    \end{align*}
    Therefore, for each $\alpha\geq 2$, we have
    \begin{equation}\label{hEwidehatE}\begin{split}
        \hE\left[\sup_{t\in[0,T]}|Y^n_t-\bar{Y}^n_t|^\alpha\right]=&\hE\left[\left(\mathcal{E}(\widehat{b}^{n,\varepsilon})_T\right)^{\frac{1}{r}}\left(\mathcal{E}(\widehat{b}^{n,\varepsilon})_T\right)^{-\frac{1}{r}}\sup_{t\in[0,T]}|Y^n_t-\bar{Y}^n_t|^\alpha\right]\\
        \leq &\left(\hE\left[\left(\mathcal{E}(\widehat{b}^{n,\varepsilon})_T\right)^{\frac{1}{1-r}}\right]\right)^{\frac{r-1}{r}}\left(\hE\left[\mathcal{E}(\widehat{b}^{n,\varepsilon})_T\sup_{t\in[0,T]}|Y^n_t-\bar{Y}^n_t|^{\alpha r}\right]\right)^{\frac{1}{r}}\\
        \leq &C_r\left(\widehat{\E}^{n,\varepsilon}\left[\sup_{t\in[0,T]}|Y^n_t-\bar{Y}^n_t|^{\alpha r}\right]\right)^{\frac{1}{r}}.
    \end{split}\end{equation}

        By Equations \eqref{widetildeYn}, \eqref{doob for widehatEnvarepsilon} and \eqref{hEwidehatE}, it suffices to prove that for any $\alpha\geq 2$
        \begin{align*}
            &\lim_{\varepsilon\rightarrow\infty}\widehat{\E}^{n,\varepsilon}\left[\left(\int_0^T|\widehat{Y}^n_s|^{\alpha-2}\widehat{Y}^n_s\widehat{m}^{n,\varepsilon}_sd\langle B\rangle_s\right)^{\beta}\right]=0, \textrm{ for any fixed } n\in\mathbb{N},\\
            &\lim_{n\rightarrow\infty}\widehat{\E}^{n,\varepsilon}\left[\left(\int_0^T|\widehat{Y}^n_s|^{\alpha-2}(Y^n_s-L_s)^-d\bar{A}^n_s\right)^{\beta}\right]=0.
        \end{align*}
        Indeed, simple calculation yields that
	\begin{align*}
		&\widehat{\E}^{n,\varepsilon}\left[\left(\int_0^T|\widehat{Y}^n_s|^{\alpha-2}(Y^n_s-L_s)^-d\bar{A}^n_s\right)^{\beta}\right]\\
		\le&\widehat{\E}^{n,\varepsilon}\left[\sup_{s\in[0,T]}|\widehat{Y}^n_s|^{(\alpha-2)\beta}\sup_{s\in[0,T]}\left((Y_s^n-L_s)^-\right)^{\beta}\left(\bar{A}^{n}_T\right)^{\beta}\right]\\
		\le&C_q\left({\hE}\left[\sup_{s\in[0,T]}|\widehat{Y}^n_s|^{(\alpha-2)\beta q}\sup_{s\in[0,T]}\left((Y_s^n-L_s)^-\right)^{\beta q}\left(\bar{A}^{n}_T\right)^{\beta q}\right]\right)^{\frac{1}{q}}\\
        \le&C_{\alpha,\beta,q}\left(\|\bar{Y}^n\|_{S^\infty_G}^{(\alpha-2)\beta q}+\|{Y}^n\|_{S^\infty_G}^{(\alpha-2)\beta q}\right)^{\frac{1}{q}}\left({\hE}\left[\sup_{s\in[0,T]}\left((Y_s^n-L_s)^-\right)^{\alpha\beta q}\right]	\right)^{\frac{1}{\alpha}}\left({\hE} \left[\left(\bar{A}^{n}_T\right)^{\alpha\beta q}\right]\right)^{\frac{1}{\alpha}}
	\end{align*}
    and 
    \begin{align*}
        &\widehat{\E}^{n,\varepsilon}\left[\left(\int_0^T|\widehat{Y}^n_s|^{\alpha-2}\widehat{Y}^n_s\widehat{m}^{n,\varepsilon}_sd\langle B\rangle_s\right)^{\beta}\right]\\
        \leq &C_{\beta}\left(\|\bar{Y}^n\|_{S^\infty_G}^{(\alpha-1)\beta }+\|{Y}^n\|_{S^\infty_G}^{(\alpha-1)\beta }\right)\widehat{\E}^{n,\varepsilon}\left[\left(\int_0^T\widehat{m}^{n,\varepsilon}_sd\langle B\rangle_s\right)^{\beta}\right]\\
        \leq&C_{\beta,q}\left(\|\bar{Y}^n\|_{S^\infty_G}^{(\alpha-1)\beta }+\|{Y}^n\|_{S^\infty_G}^{(\alpha-1)\beta }\right)\left({\hE}\left[\left(\int_0^T\widehat{m}^{n,\varepsilon}_sd\langle B\rangle_s\right)^{\beta q}\right]\right)^{\frac{1}{q}}.
    \end{align*}
    By Lemma \ref{est-barYn-barZn}, Lemma \ref{Yn-L} and \eqref{est-YnZnKn}, we obtain the desired result, i.e.,
   \begin{align*}
            &\lim_{\varepsilon\rightarrow\infty}\widehat{\E}^{n,\varepsilon}\left[\left(\int_0^T|\widehat{Y}^n_s|^{\alpha-2}\widehat{Y}^n_s\widehat{m}^{n,\varepsilon}_sd\langle B\rangle_s\right)^{\beta}\right]=0, \textrm{ for any fixed } n\in\mathbb{N},\\
            &\lim_{n\rightarrow\infty}\widehat{\E}^{n,\varepsilon}\left[\left(\int_0^T|\widehat{Y}^n_s|^{\alpha-2}(Y^n_s-L_s)^-d\bar{A}^n_s\right)^{\beta}\right]=0.
        \end{align*}
        The proof is complete.
\end{proof}

\section{Probabilistic representation for fully nonlinear PDEs with double obstacles} \label{sec:Probabilistic representation for fully nonlinear PDEs with double obstacles}

In this section, we establish the relation between doubly reflected quadratic $G$-BSDEs and obstacle problems for fully nonlinear parabolic PDEs. For simplicity, we only consider the doubly reflected BSDEs driven by $1$-dimensional $G$-Brownian motion. Similar results still holds for the multi-dimensional case. More precisely, the obstacle problem takes the following form
\begin{equation}\label{PDE}
	\begin{cases}
		\max\bigg(u-h',\min\big(-\partial_t u(t,x)-F(D_x^2 u,D_x u,u,x,t),u-h\big)\bigg)=0,
		&\\
		u(T,x)=\phi(x),  &
	\end{cases}
\end{equation}
where
\begin{align*}
	F(D_x^2 u,D_x u,u,x,t)=&G(H(D_x^2 u,D_x u,u,x,t))+ b(t,x)D_x u+f(t,x,u,\sigma(t,x)D_x u),\\
	H(D_x^2 u,D_x u,u,x,t)=&\sigma^2(t,x) D_x^2u+2l(t,x)D_xu
	+2g(t,x,u, \sigma(t,x)D_xu).
\end{align*}
Throughout this section, $b,l,\sigma,h,h':[0,T]\times\mathbb{R}\rightarrow \mathbb{R}$, $\phi:\mathbb{R}\rightarrow \mathbb{R}$ and  $f,g:[0,T]\times\mathbb{R}^3\rightarrow \mathbb{R}$  are deterministic functions and satisfy the following conditions:
\begin{description}
	\item[(Ai)] $b$, $l$, $\sigma$, $f$, $g$, $h$, $h'$ are uniformly continuous in $t$, i.e., there exists a non-decreasing continuous function $w:[0,\infty)\rightarrow[0,\infty)$ with $w(0)=0$, such that 
    \begin{align*}
        &\sup_{(x,y,z)\in\mathbb{R}^3}|l_1(t,x,y,z)-l_1(t',x,y,z)|\leq w(|t-t'|), \ l_1=f,g,\\
        &\sup_{x\in\mathbb{R}}|l_2(t,x)-l_2(t',x)|\leq w(|t-t'|), \ l_2=b,l,\sigma,h,h'.
    \end{align*}
	\item[(Aii)] There exist a positive integer $k$ and a constant $L$ such that
	\begin{align*}
		&|b(t,x)-b(t,x')|+|l(t,x)-l(t,x')|+|\sigma(t,x)-\sigma(t,x')|\leq L|x-x'|,\\
        &|h(t,x)-h(t,x')|+|h'(t,x)-h'(t,x')|\leq L|x-x'|,\\
		&|\phi(x)-\phi(x')|\leq L(1+|x|^k+|x'|^k)|x-x'|,\\
		&|f(t,x,y,z)-f(t,x',y',z')|+|g(t,x,y,z)-g(t,x',y',z')|\\
        &\leq \kappa[(1+|x|^k+|x'|^k)|x-x'|+|y-y'|+(1+|z|+|z'|)|z-z'|].
	\end{align*}
    \item[(Aiii)] There exist two positive constants $M_0,N_0$, such that 
    \begin{align*}
        &\sup_{(t,x)\in[0,T]\times \mathbb{R}}(|f(t,x,0,0)|+|g(t,x,0,0)|+|b(t,x)|+|l(t,x)|)+\sup_{x\in\mathbb{R}}|\phi(x)|\leq M_0,\\
        & h(t,x)\leq N_0, \ -h'(t,x)\leq N_0 \textrm{  for any } (t,x)\in[0,T]\times\mathbb{R}.
    \end{align*}
		\item[(Aiv)]  For any $x\in\mathbb{R}$ and $t\in[0,T]$, there exist two constants $0<\varepsilon<K$, such that $\varepsilon\leq \sigma^2(t,x)\leq K$.
	\item[(Av)] $h'\in C^{1,2}_b([0,T]\times\mathbb{R})$, where $C^{1,2}_b([0,T]\times\mathbb{R})$ refers to the space of functions that are continuously differentiable in their first variable and twice continuously differentiable in their second variable, and both derivatives are bounded. Moreover, for any $x\in\mathbb{R}$ and $t\in[0,T]$, $h(t,x)\leq h'(t,x)$ and $h(T,x)\leq \phi(x)\leq h'(T,x)$.  
\end{description}

For each $0\leq t\leq T$ and $\xi\in \cap_{p\geq 2}L_G^p(\Omega_t)$, consider the following $G$-SDE:
\begin{equation}\label{GSDE}
	X_s^{t,\xi}=\xi +\int_t^s b(r,X_r^{t,\xi})dr+\int_t^s l(r,X_r^{t,\xi})d\langle B\rangle_r+\int_t^s \sigma(r,X_r^{t,\xi})dB_r.
\end{equation}
Then, we have the following estimates, which can be seen in \cite{P19}, Exercise 5.4.8.
\begin{proposition}[\cite{P19}]\label{the1.17}
	Let $\xi,\xi'\in L_G^p(\Omega_t)$ and $p\geq 2$. Then we have, for each $\delta\in[0,T-t]$,
	\begin{align*}
		\hat{\mathbb{E}}_t\bigg[\sup_{s\in[t,t+\delta]}|X_{s}^{t,\xi}-X_{s}^{t,\xi'}|^p]&\leq C|\xi-\xi'|^p,\\
		\hat{\mathbb{E}}_t[|X_{t+\delta}^{t,\xi}|^p]&\leq C(1+|\xi|^p),\\
		\hat{\mathbb{E}}_t\bigg[\sup_{s\in[t,t+\delta]}|X_s^{t,\xi}-\xi|^p]&\leq C(1+|\xi|^p)\delta^{p/2},
	\end{align*}
	where the constant $C$ depends on $L,G,p$ and $T$.
\end{proposition}

For any fixed $(t,x)\in[0,T]\times\mathbb{R}$, consider the doubly reflected quadratic $G$-BSDE with parameters $(\xi^{t,x},f^{t,x},g^{t,x},L^{t,x}, U^{t,x})$, where
\begin{align*}
	&\xi^{t,x}=\phi(X_T^{t,x}),\  L_s^{t,x}=h(s,X_s^{t,x}), \ U_s^{t,x}=h'(s,X_s^{t,x}), \\
   & f^{t,x}(s,y,z)=f(s,X_s^{t,x},y,z), \ g^{t,x}(s,y,z)=g(s,X_s^{t,x},y,z). 
	\end{align*}
Moreover, consider the following penalized reflected $G$-BSDEs:
\begin{align}\label{Yntx}
	\begin{cases}
		Y_s^{n,t,x}= \phi(X_T^{t,x})+\int_s^T f(r,X_r^{t,x},Y_r^{n,t,x},Z_r^{n,t,x})dr+\int_s^T g(r,X_r^{t,x},Y_r^{n,t,x},Z_r^{n,t,x})d\langle B\rangle_r\\
		\ \ \ \ \ \ \ \ \ \ \ -n\int_s^T (Y_r^{n,t,x}-h'(r,X_r^{t,x}))^+dr-\int_s^T Z_r^{n,t,x}dB_r+(A_T^{n,t,x}-A_s^{n,t,x}), \quad s\in[t,T],\\
		Y_s^{n,t,x,}\geq h(s,X^{t,x}_s), s\in[t,T], \\
		\{\int_t^s (h(r,X^{t,x}_r-Y^{n,t,x}_r))dA^{n,t,x}_r\}_{s\in[t,T]} \textrm{ is a non-increasing $G$-martingale}.
\end{cases}\end{align}

    It is worth pointing out that, under Assumptions (Ai)-(Av), it is not clear if (H2) holds for $f^{t,x},g^{t,x}$ and Assumption \ref{assLU} (2) holds for $L^{t,x}$. That is, we cannot use Theorem \ref{main1} and Theorem \ref{wellposedness for RGBSDE} directly to  obtain the result that both the doubly reflected $G$-BSDE with parameters $(\xi^{t,x},f^{t,x},g^{t,x},L^{t,x}, U^{t,x})$ and the reflected $G$-BSDE \eqref{Yntx} admit a unique solution. However, for the single reflected equation \eqref{Yntx}, by Remark 6.5 in \cite{CT},  under Assumptions (Ai)-(Av), it indeed has a unique solution, denote by $(Y^{n,t,x},Z^{n,t,x},A^{n,t,x})$. We define
\begin{align}\label{untx}
	u_n(t,x):=Y_t^{n,t,x}, \quad (t,x)\in[0,T]\times \mathbb{R}.
\end{align}
    
    Actually, (H2) and Assumption \ref{assLU} (2) are only used to ensure the existence of solutions to penalized quadratic $G$-BSDEs. Fortunately, under Assumptions (Ai)-(Av), by Theorem 6.2 in \cite{CT},  the following penalized $G$-BSDE parameterized by $m$ admits a unique solution
    \begin{align*}
        Y_s^{m,t,x}= &\phi(X_T^{t,x})+\int_s^T f(r,X_r^{t,x},Y_r^{m,t,x},Z_r^{m,t,x})dr+\int_s^T g(r,X_r^{t,x},Y_r^{m,t,x},Z_r^{m,t,x})d\langle B\rangle_r\\
		& +m\int_s^T (Y_r^{m,t,x}-h(r,X^{t,x}_r))^-dr-m\int_s^T (Y_r^{m,t,x}-h'(r,X_r^{t,x}))^+dr\\
        &-\int_s^T Z_r^{m,t,x}dB_r-(K_T^{m,t,x}-K_s^{m,t,x}).
    \end{align*}
    Therefore, all the analysis and results in Section 3 still hold under (Ai)-(Av). That is, the doubly reflected $G$-BSDE with parameters $(\xi^{t,x},f^{t,x},g^{t,x},L^{t,x}, U^{t,x})$ has a unique solution, denoted by $(Y^{t,x},Z^{t,x},A^{t,x})$. Besides, for any $s\in[t,T]$, we have 
    \begin{align*}
        Y^{n,t,x}_s\downarrow Y^{t,x}_s, \textrm{ as }n\rightarrow\infty.
    \end{align*}
    We now define
\begin{equation}\label{utx}
	u(t,x):=Y_t^{t,x},\quad (t,x)\in[0,T]\times\mathbb{R}.
\end{equation}
Similar as Remark 4.3 in \cite{HJPS2}, $u$ and $u_n$ are deterministic function. Moreover, for each fixed $(t,x)$, we have $u_n(t,x)\downarrow u(t,x)$ and by the result in \cite{CT}, $u_n$ is a continuous function.

Our main result in this section is to show that $u$ defined by \eqref{utx} is the viscosity solution to the fully nonlinear obstacle problem \eqref{PDE}. We first give the definition of viscosity solution for \eqref{PDE}, which is based on the notions of sub-jets and super-jets. For more details, we may refer to the paper \cite{CIL}. 

\begin{definition}
	Let $u\in C((0,T)\times\mathbb{R})$ and $(t,x)\in(0,T)\times \mathbb{R}$. We denote by $\mathcal{P}^{2,+} u(t,x)$ (the ``parabolic superjet" of $u$ at $(t,x)$) the set of triples $(p,q,X)\in\mathbb{R}^3$ satisfying
	\begin{align*}
		u(s,y)\leq u(t,x)+p(s-t)+ q(y-x)+\frac{1}{2} X(y-x)^2+o(|s-t|+|y-x|^2).
	\end{align*}
	Similarly, we define $\mathcal{P}^{2,-} u(t,x)$ (the ``parabolic subjet" of $u$ at $(t,x)$) by $\mathcal{P}^{2,-} u(t,x):=-\mathcal{P}^{2,+}(- u)(t,x)$.
\end{definition}


\begin{definition}
	Let $u$ be a continuous function defined on $[0,T]\times \mathbb{R}$. It is called a viscosity:\\
	\noindent (i) subsolution  of \eqref{PDE} if $u(T,x)\leq \phi(x)$, $x\in\mathbb{R}$, and at any point $(t,x)\in(0,T)\times\mathbb{R}$, for any $(p,q,X)\in\mathcal{P}^{2,+}u(t,x)$,
	\begin{equation}\label{viscosity sub}
		\max\bigg(u(t,x)-h'(t,x), \min\big(u(t,x)-h(t,x), -p-F(X,q,u(t,x),x,t)\big)\bigg)\leq 0;
	\end{equation}
	(ii) supersolution of \eqref{PDE} if $u(T,x)\geq\phi(x)$, $x\in\mathbb{R}$, and at any point $(t,x)\in(0,T)\times\mathbb{R}$, for any $(p,q,X)\in\mathcal{P}^{2,-}u(t,x)$,
	\begin{equation}\label{viscosity super}
		\max\bigg(u(t,x)-h'(t,x),\min\big(u(t,x)-h(t,x), -p-F(X,q,u(t,x),x,t)\big)\bigg)\geq 0;
	\end{equation}
	(iii) solution of \eqref{PDE} if it is both a viscosity subsolution and supersolution.
\end{definition}

By the above definition, a viscosity solution should be continuous. In the sequel, we first show that $u$ defined by \eqref{utx} is indeed a continuous function. For simplicity, we only consider the case where $f=0$. The results still hold for the other cases. 


\begin{lemma}\label{l1}
	We have $u\in C([0,T]\times \mathbb{R}). $
\end{lemma}

\begin{proof}
	First, for any fixed $t\in[0,T]$, we show that  $u$ is a continuous function in $x$. By Proposition \ref{est-Y1-Y2}, Proposition \ref{the1.17} and the proof of Theorem \ref{main1}, noting that $u$ is deterministic, for any $t\in[0,T]$ and any $x_1,x_2\in\mathbb{R}$, there exists a constant $C$ depending on $T,k,L,G,x_1,x_2,M_0,N_0$, such that
	\begin{align*}
        &|u(t,x_1)-u(t,x_2)|^2\\
        \leq &C\left(\hE\left[|\phi(X^{t,x_1}_T)-\phi(X^{t,x_2}_T)|^{2p}\right]\right)^{\frac{1}{p}}+C\left(\hE\left[\left(\int_t^T|\hat{\lambda}|^2 d\langle B \rangle_s\right)^p\right]\right)^{\frac{1}{p}}\\
       &+C\left(\hat{\mathbb{E}}\left[\sup_{s\in[t,T]}|L^{t,x_1}_s-L^{t,x_2}_s|^{2p}\right]\right)^{\frac{1}{2p}}+C\left(\hat{\mathbb{E}}\left[\sup_{s\in[t,T]}|U^{t,x_1}_s-U^{t,x_2}_s|^{2p}\right]\right)^{\frac{1}{2p}}\\
        \leq &C\left(\hE\left[\left(1+|X^{t,x_1}_T|^{2pk}+|X^{t,x_2}_T|^{2pk}\right)^2\right]\right)^{\frac{1}{2p}}\left(\hE\left[|X^{t,x_1}_T-X^{t,x_2}_T|^{4p}\right]\right)^{\frac{1}{2p}}\\
        &+C\left(\hE\left[\int_t^T\left(1+|X^{t,x_1}_s|^{2pk}+|X^{t,x_2}_s|^{2pk}\right)^2ds\right]\right)^{\frac{1}{2p}}\left(\hE\left[\int_t^T|X^{t,x_1}_s-X^{t,x_2}_s|^{4p}ds\right]\right)^{\frac{1}{2p}}\\
        &+C\left(\hE\left[\sup_{s\in[t,T]}|X^{t,x_1}_s-X^{t,x_2}_s|^{2p}\right]\right)^{\frac{1}{2p}}\\
        \leq & C(1+|x_1|^{2k}+|x_2|^{2k})|x_1-x_2|^2+C|x_1-x_2|,
    \end{align*}
    where $\hat{\lambda}_s=g(s,X^{t,x_1}_s,Y^{t,x_2}_s,Z^{t,x_2}_s)-g(s,X^{t,x_2}_s,Y^{t,x_2}_s,Z^{t,x_2}_s)$. 	

    Now, we are in a position to show that  any fixed $x\in\mathbb{R}$, $u$ is continuous in $t$. For any fixed $t\in[0,T]$, we define $X_s^{t,x}:=x$, $Y_s^{t,x}:=Y_t^{t,x}$, $Z_s^{t,x}:= 0$, $A_s^{t,x}:=0$, $U^{t,x}_s:=h'(t,x)$ and $L^{t,x}_s:=h(t,x)$ for $0\leq s\leq t$. Obviously, $(Y^{t,x}_s,Z^{t,x}_s,A^{t,x}_s)_{s\in[0,T]}$ is the solution to the doubly reflected $G$-BSDE with parameters $(\phi(X^{t,x}_T), \tilde{g}^{t,x}, L^{t,x}, U^{t,x})$, where $\tilde{g}^{t,x}(s,y,z)=\mathbf{1}_{[t,T]}(s)g(s,X^{t,x}_s,y,z)$. For each fixed $x\in\mathbb{R}$, suppose that $0\leq t_1\leq t_2\leq T$, by Proposition \ref{est-Y1-Y2}, for some $p\geq 2$, there exists a constant $C$ depending on $T,k,L,G,x,p$, such that
	\begin{align*}
		&|u(t_1,x)-u(t_2,x)|^2=|Y^{t_1,x}_0-Y^{t_2,x}_0|^2\\
		\leq &C\left(\mathbb{E}\left[\sup_{t\in[0,T]}|L^{t_1,x}_t-L^{t_2,x}_t|^{2p}\right]+\mathbb{E}\left[\sup_{t\in[0,T]}|U^{t_1,x}_t-U^{t_2,x}_t|^{2p}\right]\right)^{\frac{1}{2p}}+C\left(\mathbb{E}\left[|\phi(X_T^{t_1,x})-\phi(X_T^{t_2,x})|^{2p}\right]\right)^{\frac{1}{p}}\\
		&+C\left(\mathbb{E}\left[\left(\int_0^T |\tilde{g}^{t_1,x}(s,X^{t_1,x}_s,Y^{t_2,x}_s,Z^{t_2,x}_s)-\tilde{g}^{t_2,x}(s,X^{t_2,x}_s,Y^{t_2,x}_s,Z^{t_2,x}_s)|^2d\langle B\rangle_s\right)^p\right]\right)^{\frac{1}{p}}.
	\end{align*}
	Note that 
	\begin{align*}
		&\sup_{t\in[0,T]}|L^{t_1,x}_t-L^{t_2,x}_t|\\
		\leq&|h(t_1,x)-h(t_2,x)|+\sup_{t\in[t_1,t_2]}|h(t,X_t^{t_1,x})-h(t_2,x)|+\sup_{t\in[t_2,T]}|h(t,X^{t_1,x}_t)-h(t,X^{t_2,x}_t)|\\
		\leq &2\sup_{t\in[t_1,t_2]}|h(t,x)-h(t_2,x)|+\sup_{t\in[t_1,t_2]}|h(t,X^{t_1,x}_t)-h(t,x)|+\sup_{t\in[t_2,T]}L|X^{t_1,x}_t-X^{t_2,x}_t|\\
		\leq & 2\sup_{t\in[t_1,t_2]}|h(t,x)-h(t_2,x)|+\sup_{t\in[t_1,t_2]}L|X^{t_1,x}_t-x|+\sup_{t\in[t_2,T]}L|X^{t_2,X^{t_1,x}_{t_2}}_t-X^{t_2,x}_t|
	\end{align*}
    and 
    \begin{align*}
        |\phi(X_T^{t_1,x})-\phi(X_T^{t_2,x})|&\leq C(1+|X^{t_1,x}_T|^k+|X^{t_2,X}_T|^k)|X^{t_1,x}_T-X^{t_2,x}_T|\\
        &= C(1+|X^{t_1,x}_T|^k+|X^{t_2,X}_T|^k)|X^{t_2,X^{t_1,x}_{t_2}}_T-X^{t_2,x}_T|.
    \end{align*}
	Letting $\delta=t_2-t_1$, by Proposition \ref{the1.17}, we have
	\begin{align*}
		&\lim_{\delta\rightarrow 0}\mathbb{E}\left[\sup_{t\in[0,T]}|L^{t_1,x}_t-L^{t_2,x}_t|^{2p}\right]=0,\\
        &\lim_{\delta\rightarrow 0}\mathbb{E}\left[|\phi(X^{t_1,x}_T)-\phi(X^{t_2,x}_T)|^{2p}\right]=0.
	\end{align*}
	Similarly, we have
	\begin{align*}
		&\lim_{\delta\rightarrow 0}\mathbb{E}\left[\sup_{t\in[0,T]}|U^{t_1,x}_t-U^{t_2,x}_t|^{2p}\right]=0.
	\end{align*}
	By simple calculation, we obtain that
	\begin{align*}
		&\int_0^T |\tilde{g}^{t_1,x}(s,X^{t_1,x}_s,Y^{t_2,x}_s,Z^{t_2,x}_s)-\tilde{g}^{t_2,x}(s,X^{t_2,x}_s,Y^{t_2,x}_s,Z^{t_2,x}_s)|^2d\langle  B\rangle_s\\
		\leq &C \int_{t_1}^{t_2}\bigg(|g(s,0,0,0)|^2+|X_s^{t_1,x}|^{2k+2}+|X_s^{t_1,x}|^2+|Y_s^{t_2,x}|^2\bigg)ds\\
		&+C\int_{t_2}^T \bigg(1+|X_s^{t_1,x}|^k+|X_s^{t_2,x}|^k\bigg)^2|X_s^{t_1,x}-X^{t_2,x}_s|^2ds.
	\end{align*}
	Applying Proposition \ref{the1.17} and the fact that $Y^{t_2,x}\in S^\infty_G(0,T)$, we have
	\begin{displaymath}
		\lim_{\delta\rightarrow 0}\mathbb{E}\left[\left(\int_0^T |\tilde{g}^{t_1,x}(s,X^{t_1,x}_s,Y^{t_2,x}_s,Z^{t_2,x}_s)-\tilde{g}^{t_2,x}(s,X^{t_2,x}_s,Y^{t_2,x}_s,Z^{t_2,x}_s)|^2d\langle B\rangle_s\right)^p\right]=0.
	\end{displaymath}
	All the above analysis implies that $u$ is continuous in $t$. The proof is complete.
\end{proof}

Now, we present the main result in this section. 

\begin{theorem}\label{the1.21}
	The function $u$ defined by \eqref{utx} is a viscosity solution of the double obstacle problem \eqref{PDE}.
\end{theorem}

\begin{proof}
    Recalling $u_n$ defined by \eqref{untx}, by Theorem 6.4 in \cite{CT}, it is the viscosity solution of the following  PDE
\begin{equation}\label{eq1.9}
	\begin{cases}
		\min\bigg(u_n(t,x)-h(t,x),-\partial_t u_n-F_n(D_x^2 u_n,D_x u_n,u_n,x,t)\bigg)=0, & (t,x)\in[0,T]\times\mathbb{R}\\
		u_n(T,x)=\phi(x), &  x\in\mathbb{R},
\end{cases}\end{equation}
where
\begin{displaymath}
	F_n(D_x^2 u,D_x u,u,x,t)=F(D_x^2 u,D_x u,u,x,t)-n(u(t,x)-h'(t,x))^+.
\end{displaymath}

	Note the facts that $u_n$ is continuous (see Lemma 6.2 and Lemma 6.3 in \cite{CT}) and for each $(t,x)\in[0,T]\times\mathbb{R}$, we have
	\begin{displaymath}
		u_n(t,x)\downarrow u(t,x), \textrm{ as }n\rightarrow \infty.
	\end{displaymath}
	 By Lemma \ref{l1} and Dini's theorem, the sequence $u^n$  converges uniformly to $u$ on compact sets. The proof will be divided into two parts.
	
	\textbf{Step 1: subsolution}. For each fixed $(t,x)\in(0,T)\times\mathbb{R}$, let $(p,q,X)\in\mathcal{P}^{2,+} u(t,x)$. For the case that $u(t,x)=h(t,x)$,  since we have $u(t,x)\leq h'(t,x)$, \eqref{viscosity sub} clearly holds. 
	Now suppose that $u(t,x)>h(t,x)$. By Lemma 6.1 in \cite{CIL}, there exist sequences
	\begin{displaymath}
		n_j\rightarrow \infty,\ (t_j,x_j)\rightarrow (t,x),\ (p_j,q_j,X_j)\in \mathcal{P}^{2,+} u_{n_j}(t_j,x_j),
	\end{displaymath}
	such that $(p_j,q_j,X_j)\rightarrow(p,q,X)$. Recalling that $u_n$ is the viscosity solution to PDE \eqref{eq1.9}, hence a subsolution, we have, for any $j$,
	\begin{displaymath}
		-p_j-F(X_j,q_j,u_{n_j}(t_j,x_j),x_j,t_j)\leq -n_{j}(u_{n_j}(t_j,x_j)-h'(t_j,x_j))^+\leq 0.
	\end{displaymath}
	Letting $j$ go to infinity in the above inequality yields that 
	\begin{displaymath}
		-p-F(X,q,u(t,x),x,t)\leq 0.
	\end{displaymath}
    All the above analysis indicates that \eqref{viscosity sub} holds. 	Therefore,  $u$ is a subsolution of \eqref{PDE}.
	
	\textbf{Step 2: supersolution.} For each fixed $(t,x)\in(0,T)\times\mathbb{R}$, let $(p,q,X)\in\mathcal{P}^{2,-} u(t,x)$. For the case that $u(t,x)=h'(t,x)$, \eqref{viscosity super} clearly holds. It remains to show that when $u(t,x)<h'(t,x)$, \eqref{viscosity super} holds, i.e.,
	\begin{displaymath}
		-p-F(X,q,u(t,x),x,t)\geq 0.
	\end{displaymath}
	Applying Lemma 6.1 in \cite{CIL} again, there exist sequences
	\begin{displaymath}
		n_j\rightarrow \infty,\ (t_j,x_j)\rightarrow (t,x),\ (p_j,q_j,X_j)\in \mathcal{P}^{2,-} u_{n_j}(t_j,x_j),
	\end{displaymath}
	such that $(p_j,q_j,X_j)\rightarrow(p,q,X)$. Since $u_n$ is the viscosity solution to equation \eqref{eq1.9}, hence a supersolution, we derive that for any $j$,
	\begin{displaymath}
		-p_j-F_{n_j}(X_j,q_j,u_{n_j}(t_j,x_j),x_j,t_j)\geq 0.
	\end{displaymath}
	Recalling that $u_n$ converges uniformly on compact sets and $u(t,x)<h'(t,x)$, for $j$ sufficiently large, we have $u_{n_j}(t_j,x_j)<h'(t_j,x_j)$. Then, we obtain the desired result by letting $j$ approach infinity in the above inequality. Thus, $u$ is a viscosity solution of \eqref{PDE}.	 
\end{proof}

\section*{Acknowledgments}
	Li's research was supported  by the National Natural Science Foundation of China (No. 12301178), the Natural Science Foundation of Shandong Province for Excellent Young Scientists Fund Program (Overseas) (No. 2023HWYQ-049), the Fundamental Research Funds for the Central Universities  and  the Qilu Young Scholars Program of Shandong University. Luo's research was supported by the National Natural Science Foundation of China ( No. 12326603, No. W2511002).


\end{document}